\documentclass[12pt,a4paper]{amsart}
\usepackage{hyperref}
\usepackage{amsmath}
\usepackage{amssymb}
\usepackage{amsthm}
\usepackage[all]{xy}
\usepackage{xcolor}
\usepackage{tikz, tikz-cd}

\newcommand{\B}{\ensuremath{\mathcal{B}}}

\newcommand{\dd}{\textrm{d}}
\newcommand{\E}{\ensuremath{\mathcal{E}}}
\newcommand{\F}{\ensuremath{\mathcal{F}}}

\newcommand{\K}{\ensuremath{\mathcal{K}}}
\newcommand{\LLL}{\ensuremath{\mathcal{L}}}

\newcommand{\N}{\ensuremath{\mathcal{N}}}

\newcommand{\oo}{\ensuremath{\mathcal{O}}}

\newcommand{\R}{\ensuremath{\mathcal{R}}}
\newcommand{\SSS}{\ensuremath{\mathcal{S}}}

\newcommand{\V}{\ensuremath{\mathcal{V}}}

\newcommand{\X}{\ensuremath{\mathcal{X}}}

\newcommand{\CC}{\ensuremath{\mathbb{C}}}
\newcommand{\FF}{\ensuremath{\mathbb{F}}}

\newcommand{\PP}{\ensuremath{\mathbb{P}}}

\newcommand{\QQ}{\ensuremath{\mathbb{Q}}}
\newcommand{\RR}{\ensuremath{\mathbb{R}}}
\newcommand{\ZZ}{\ensuremath{\mathbb{Z}}}

\newcommand{\s}{\ensuremath{\mathfrak{s}}}

\def\hol{{\mathcal{O}}}

\DeclareMathOperator{\Cl}{Cl}

\DeclareMathOperator{\Ext}{Ext}

\DeclareMathOperator{\Hom}{Hom}

\DeclareMathOperator{\NE}{\overline{NE}}
\DeclareMathOperator{\Pic}{Pic}
\DeclareMathOperator{\Proj}{\mathbf Proj}
\DeclareMathOperator{\Sym}{Sym}
\DeclareMathOperator{\wSym}{wSym}
\DeclareMathOperator{\Sing}{Sing}

\DeclareMathOperator{\lcm}{lcm}
\DeclareMathOperator{\rank}{rank}

\DeclareMathOperator{\hcf}{hcf}

\newtheorem{teo}{Theorem}[section]
\newtheorem{lem}[teo]{Lemma}

\newtheorem{cor}[teo]{Corollary}
\newtheorem{prop}[teo]{Proposition}

\theoremstyle{definition}
\newtheorem{df}[teo]{Definition}

\theoremstyle{remark}
\newtheorem{remark}[teo]{Remark}
\newtheorem{example}[teo]{Example}

\title{Simple fibrations in $(1,2)$-surfaces}
\author{Stephen Coughlan \and Roberto Pignatelli}
\address{Stephen Coughlan (former)\\Mathematisches Institut \\Lehrstuhl Mathematik VIII
\\ Universit\"atsstra\ss e 30 \\ 95447 Bayreuth
\\Germany}
\email{stephen.coughlan@uni-bayreuth.de}
\address{Stephen Coughlan (current)\\Institute of Mathematics of the Polish Academy of Sciences\\ul. \'Sniadeckich 8\\P.O. Box 21\\00-656 Warszawa\\Poland}
\email{scoughlan@impan.pl}

\address{Roberto Pignatelli\\Dipartimento di Matematica\\ Universit\`a di Trento\\ Via Sommarive 14\\
loc. Povo\\ I-38123 Trento\\ Italy}
\email{roberto.pignatelli@unitn.it}

\begin{document}

\begin{abstract}
We introduce the notion of a simple fibration in $(1,2)$-surfaces. That is, a
hypersurface inside a certain weighted projective space bundle over a curve such that the general fibre is a minimal surface of general type with $p_g=2$ and $K^2=1$.
We prove that almost all Gorenstein simple fibrations over the projective line 
with at worst canonical singularities are canonical threefolds  ``on the Noether 
line" with $K^3=\frac43 p_g-\frac{10}3$, and we classify them. Among them, we 
find all the canonical threefolds on the Noether line that have previously 
appeared in the literature.

The Gorenstein simple fibrations over $\PP^1$ are Cartier divisors in a toric $4$-fold. This allows to us to show among other things, that the previously known canonical threefolds on the Noether line form an open subset of the moduli space of canonical threefolds, that the general element of this component is a Mori Dream Space, and that there is a second component when the geometric genus is congruent to $6$ modulo $8$; the threefolds in this component are new.
\end{abstract}

\maketitle

\tableofcontents

\section*{Introduction}
A \emph{$(1,2)$-surface} $S$ is a minimal surface of general type with invariants $p_g=2$, $q=0$, $K^2=1$. These surfaces are classified in \cite[Theorem 2.1]{hor2}  as double covers of the weighted projective space $\PP(1,1,2)$ (equivalently the quadric cone), branched over a curve of weighted degree ten and also over the singular point $(0,0,1)$. Their canonical model is a hypersurface of weighted degree ten in $\PP(1,1,2,5)$, with at worst rational double points as singularities (compare \cite[Theorem 3.3]{FPR} where this known result is generalized to Gorenstein stable surfaces).

These surfaces lie at the heart of the recent progress in the study of threefolds of general type (see for example \cite{CCJ,CCZ,HZ1}).
Namely, it seems that the threefolds that are fibred in $(1,2)$-surfaces are somewhat analogous to the genus $2$ fibrations in the theory of surfaces of general type.

There is now a satisfactory theory of surfaces with a genus $2$ fibration (see, e.g.~\cite{HorPen, Xiao, reid90, pencils, big}). A key feature of genus $2$ fibrations is that the singular fibres may have several different topological types (see \cite{Ogg66}) but despite this, they fit ``algebraically'' into just two classes:  the canonical ring of a genus $2$ fibre is generated by three or four elements, according to whether the fibre is $2$-connected or not. It would be nice to have a similar theory for threefolds fibred in $(1,2)$-surfaces, but the reality is much more complicated. Indeed, the study of surfaces fibred in curves of genus $g\geq 3$ is already much more difficult (see \cite{AshikagaKonno, reid90}).

This paper originated from the observation (\cite{hor1,pencils}) that the minimal surfaces of general type fulfilling the Noether equality $K_S^2=2p_g-4$ are exactly those with a genus $2$ fibration $f \colon S \rightarrow \PP^1$ such that all fibres are $2$-connected; in other words such that all fibres look like smooth fibres from the point of view of the generation of the canonical ring. This motivates the concept of {\it simple fibrations in $(1,2)$-surfaces} (see Definition \ref{def!simple}); these are threefolds $X$ with canonical singularities and a morphism $\pi \colon X \rightarrow B$ where the relative canonical class is ample and $B$ is a smooth curve such that the canonical ring of each fibre is {\it algebraically} like the ring of a $(1,2)$-surface.

In this paper, we develop a systematic theory of these simple fibrations. They have a natural description as hypersurfaces in $\PP(1,1,2,5)$-bundles over the base curve $B$; in particular, we have a
classification of all simple fibrations over $\PP^1$ as Cartier divisors in some toric $4$-fold (Theorem \ref{thm!main?}).
They are denoted by $X(d;d_0)$ in the following, and they have geometric genus $p_g=3d-2$ and canonical volume $K^3=4d-6$, in particular
\[
K_X^3=\frac43 p_g-\frac{10}3.
\]
The toric $4$-fold depends on two non-negative integers: $d$, that is related to $p_g$ by the formula above, and $d_0$, that may be any integer from $\frac{d}4$ to $\frac32d$.

Indeed the Noether inequality $K_X^3\geq \frac43 p_g-\frac{10}3$ has recently 
been proven \cite{CCJ}, excepting possibly threefolds with $5\le p_g\le10$. It 
is not known if these exceptions exist.
 The threefolds for which the equality holds are said to be {\it on the Noether line}, so our $X(d;d_0)$ are canonical models of threefolds on the Noether line.

There are other works about threefolds on the Noether line, some of which appeared during the development of this project, which started in 2015. Kobayashi \cite{Kobayashi} discovered infinitely many families of threefolds on the Noether line. These are constructed by taking the minimal model of a certain genus two fibration over a Hirzebruch surface. Kobayashi's construction was generalised by Chen and Hu \cite{CH}, who claimed a classification of smooth canonical threefolds on the Noether line for $p_g \geq 7$. Their threefolds correspond to our $X(d;d_0)$ with $d \leq d_0$. In fact, those $X(d;d_0)$  with $d>d_0$ are singular, unless $d$ is divisible by $8$ and $7d=8d_0$, in which case the general $X\left( d; \frac78 d\right)$ is (rather surprisingly) smooth!

Using our description as divisors in a toric variety we could prove among other things
\begin{teo}\label{thm!main-intro}
\begin{enumerate}
\item The canonical $3$-folds constructed by Kobayashi--Chen--Hu form an open 
subset of an unirational component of the moduli space of canonical $3$-folds 
with $K_X^3=\frac43 p_g-\frac{10}3$ for all $p_g\geq 7$ (Propositions 
\ref{prop!mainstream} and \ref{prop!BoundMainStream}).
\item The general $3$-fold in this component is a Mori Dream Space (Theorem \ref{thm!MDS}).
\item Suppose that $p_g \geq 22$ is of the form $3d-2$ with $d$ divisible by
$8$. Then the moduli space of canonical $3$-folds with $K_X^3=\frac43 
p_g-\frac{10}3$ contains a second component whose general element is smooth, 
and which includes our threefolds $X\left(d;\frac78 d\right)$ (Theorem 
\ref{thm!2comps}).
\end{enumerate}
\end{teo}

Parts 1) and 3) of this theorem look very similar to Horikawa's famous classification of the minimal surfaces of general type {\it on the Noether line} \cite[Theorems 3.3 and 7.1]{hor1}. The moduli space of Horikawa surfaces with $K^2$ divisible by $8$ has two unirational, irreducible, connected components while that of surfaces with $K^2$ not divisible by $8$ has  just one. For threefolds, when the two components arise, they actually do intersect; more precisely we produce a canonical threefold with a curve of singularities, which lies in the intersection of both irreducible components.

By analogy with Horikawa's mentioned results, we conjecture that all threefolds on the Noether line are in our list for $p_g$ sufficiently large. Then we would have as in Horikawa's case one or two irreducible components with a smooth element in it, and a complete description of the moduli space should be obtained exploiting our classification in Theorem \ref{thm!main?}.

This conjecture is supported by the recent results of \cite{HZ}, where it has 
been proven that all canonical threefolds on the Noether line are Gorenstein. 
Moreover, \cite{HZ} also determine two further lines which lie above but 
parallel to the Noether line, which they call the second and third Noether 
lines. If $p_g\geq 11$, then all canonical threefolds which do not lie on the 
Noether line, lie on or above the second Noether line, and analogously 
threefolds above the second line lie on or above the third one. In fact simple 
fibrations in $(1,2)$-surfaces over $\PP^1$ may be non-Gorenstein, in which case 
(for the sake of simplicity we suppose that $B=\PP^1$, see Proposition 
\ref{prop!L2} for the full statement) the general simple fibration has $N$ 
isolated quotient $\frac12(1,1,1)$ singularities and 
$K^3=\frac43p_g-\frac{10}3+\frac{N}6$. When $N=1$ and $2$ we get the two lines 
in \cite{HZ}. So an explanation for their result could be that for $p_g$ big 
enough and $K^3 \leq \frac43p_g-\frac{10}3+\epsilon$ (for some positive 
$\epsilon$) all canonical threefolds are simple fibrations in $(1,2)$-surfaces.

We also mention that \cite{HZ}  proved that the canonical image of a canonical threefold on the Noether line is smooth for $p_g\geq 23$, but could not determine if their bound is sharp. Our construction shows that their result is sharp, because $X(8;2)$ has $p_g=22$ and canonical image a cone, see Example \ref{exa!X(d;2)}.

The paper is organized as follows.

Section \ref{sec!Gorenstein} is devoted to the production of canonical threefolds on the Noether line. For the convenience of the reader, we describe them directly as Cartier divisors in a suitable linear system in a specific toric $4$-fold. The construction is then very explicit, depending on two integers $d$, $d_0$. The main result is the already mentioned Theorem \ref{thm!main?} giving a complete classification of Gorenstein simple fibrations in $(1,2)$-surfaces. We determine their singularities and numerical invariants according to the values of $d,d_0$. The canonical image is the Hirzebruch surface $\FF_e$ with $e=3d-2d_0$. The dichotomy of Theorem \ref{thm!main-intro} emerges here, as we find smooth examples with $e \leq d$ and with $e=\frac54 d$.
Finally, we show that in the first case the general $X(d;d_0)$ is a Mori Dream Space.

In section \ref{sec!moduli} we study the deformation theory of those $X(d;d_0)$ with $e \leq d$, showing that they form a single unirational family, whose general element has $e=0$ or $1$ according to the parity of $p_g$. This family covers an open dense subset of one irreducible component of the moduli space.

In section \ref{sec!bundles} we develop the basics of the theory of weighted projective bundles over a nonsingular base $B$. This is a natural generalization of the standard theory of $\PP^n$-bundles $\PP(\E)\to B$ where $\E$ is a vector bundle over $B$. In particular, Proposition \ref{prop!cotangent-PP} provides a relative Euler sequence for weighted projective bundles and a formula for the relative canonical sheaf.

In section \ref{sec!simple} we finally give a definition of simple fibrations in $(1,2)$-surfaces, showing that their relative canonical algebra embeds them as a divisor in a bundle in weighted projective spaces $\PP(1,1,2,5)$. Then we compute their invariants and show that if they are regular and Gorenstein, then they can be embedded in a toric $4$-fold, giving the threefolds considered in section \ref{sec!Gorenstein}.

We complete the proof of Theorem \ref{thm!main-intro} in section \ref{sec!more-noether}. Here we first compare our simple fibrations in $(1,2)$-surfaces with the Kobayashi--Chen--Hu construction, in the cases where the two coincide. Essentially, the Kobayashi--Chen--Hu model is the blowup of the base curve in $|K_X|$. Then we consider the case $7d=8d_0$ and show that these threefolds are not degenerations of threefolds given by the Kobayashi--Chen--Hu construction, although we do find a common singular degeneration with canonical singularities.

In section \ref{sec!nef-big} we finish our classification of simple fibrations over $\PP^1$ by studying a handful of special cases whose canonical class is not ample. After applying the minimal model program, we find three canonical threefolds
with $p_g=4,7,10$ respectively, which lie above the Noether line but extremely close to it; the last two appeared already recently in the literature in \cite{CJL} by a totally different construction, whereas the first one appears to be new.

\section*{Acknowledgements}
We started this project many years ago, in 2015, at a conference in Bayreuth celebrating Fabrizio Catanese's 65th birthday.
We find it appropriate to mention him here. In fact, several of the techniques used here, may be traced back to his work and his teaching. For example, as is evident from the frequent references to \cite{pencils}, we were inspired by the study of the relative canonical algebra of fibrations of surfaces to curves, which was developed by Catanese and many other people, and taught by him to the second author during his doctoral studies.
We thank him for the enthusiasm, guidance and support that he has provided to us and to the mathematical community.

The second author would like to thank Alex Massarenti for inspiring discussions on Mori Dream Spaces (in particular pointing him to the beautiful result \cite[Corollary 4.1.1.5]{ADHL}) without which we would not have obtained Theorem \ref{thm!MDS}. Moreover, we thank Bal\'{a}zs Szendr\H{o}i  for his comments on the proof.

We also thank Gavin Brown, Meng Chen, Yifan Chen, Yong Hu, Chen Jiang and Tong Zhang for helpful comments on a previous draft of this paper.

The second author was partially supported by INdAM (GNSAGA) and by MIUR PRIN 2017 “Moduli Theory and Birational Classification

\section{Threefolds on the Noether line}\label{sec!Gorenstein}
In this section we introduce and classify the simple fibrations in $(1,2)$-surfaces that are regular and Gorenstein, we show that (apart from a few exceptions) they are canonical threefolds on the Noether line.

\subsection{Toric bundles}
Choose integers $d,d_0$ and define $\FF=\FF(d;d_0)$ to be the toric 4-fold with weight matrix
\begin{equation}\label{formula!weightmatrix}
\begin{pmatrix}
t_0 & t_1 & x_0 & x_1 & y & z\\
1 & 1 & d-d_0 & d_0-2d & 0 & 0\\
0 & 0 &  1 &      1 &    2 &  5
\end{pmatrix}
\end{equation}
and irrelevant ideal $I=(t_0,t_1)\cap(x_0,x_1,y,z)$.
In other words $(\CC^*)^2$ acts on $\CC^6$ with coordinates
$t_0,t_1,x_0,x_1,y,z$ via
(\ref{formula!weightmatrix}):
\[(\lambda,\mu)\cdot(t_0,t_1,x_0,x_1,y,z)
=(\lambda t_0,\lambda t_1,\lambda^{d-d_0}\mu x_0,\lambda^{d_0-2d}\mu x_1,\mu^2y,\mu^5z)
\]
and $\FF$ is the quotient $(\CC^6\smallsetminus V(I))/(\CC^*)^2$.

Up to exchanging the $x_j$ we may and do assume without loss of generality any of the following equivalent conditions:
\begin{equation*}
d-d_0\geq d_0-2d \Longleftrightarrow d_0 \leq \tfrac32 d \Longleftrightarrow e:=3d-2d_0\ge 0.
\end{equation*}

The divisor class group $\Cl(\FF)$ is isomorphic to
$\ZZ^2$ (\cite[\S 5.1]{CLS}). We choose generators $F,H$ defined respectively by $t_0$ and
$t_0^{d_0}x_0$. With this choice, the tautological sheaf $\oo_\FF(1)$ has class $H-dF$.

Each of the {\it coordinates} $\rho \in \{t_0,t_1,x_0,x_1,y,z\}$ corresponds to
a torus invariant irreducible Weil divisor $D_\rho$ in $\FF$ whose class is as
follows
\begin{gather*}
[D_{t_0}]=[D_{t_1}]=F,\quad \quad
[D_{x_0}]=H-d_0F,\quad \quad
[D_{x_1}]=H+(d_0-3d)F,\\
[D_y]=2(H-dF),\quad \quad
[D_z]=5(H-dF).
\end{gather*}
Note that $D_y\cap D_z$ is a Hirzebruch surface $\FF_e$.

\begin{prop}\label{prop!omega_F}
$\omega_{\FF(d;d_0)} \cong \hol_{\FF(d;d_0)} (-9H+(10d-2)F)$.
\end{prop}
\begin{proof}
We have $[K_{\FF}]=-[D_{t_0} + D_{t_1}+D_{x_0} + D_{x_1} +  D_y + D_z]$
by \cite[Thm 8.2.3]{CLS}.
\end{proof}

\begin{lem} The intersection numbers on $\FF(d;d_0)$ are
\begin{gather*}
H^4=\frac{d}2,\quad\quad H^3F=\frac1{10},\quad\quad F^2=0.
\end{gather*}
\end{lem}
\begin{proof}
Clearly $F^2=0$ because any two distinct fibres are disjoint.
Since the intersection $D_{t_0}\cap D_{x_0} \cap D_y \cap D_z$ is a reduced smooth point,
$D_{t_0} D_{x_0}  D_y  D_z=10H^3F=1$.
Similarly $D_{x_0} \cap D_{x_1} \cap  D_y \cap  D_z$ is empty, so
\[
D_{x_0} D_{x_1}  D_y  D_z=10 H^4+(10 \cdot (-d_0+d_0-3d) - 5 \cdot 2d - 2 \cdot 5d)H^3F=0.
\]
Rearranging and substituting $H^3F=\frac1{10}$ gives $H^4=\frac d2$.
\end{proof}

\begin{prop}\label{prop!ample} The numerical divisor class
$aH+bF$ is
\begin{enumerate}
\item nef if and only if $a \geq 0$ and $b\geq -a\min(d,d_0)$;
\item ample if and only if $a > 0$ and $b> -a\min(d,d_0)$.
\end{enumerate}

\end{prop}
\begin{proof}
By \cite[Thms 6.3.12 and 6.3.13]{CLS} $aH+bF$ is nef (resp.~ample)  if and
only if its restriction to any torus invariant irreducible
curve is nonnegative (resp.~positive). Torus invariant irreducible
curves on $\FF$ are intersections of three of the divisors $D_{\rho}$.

The Proposition then follows from
\begin{align*}
(aH+bF)D_{t_0}D_yD_z&=10aH^3F=a,\\
(aH+bF)D_{x_1}D_yD_z&=10(aH^4+(b-a(5d-d_0)))H^3F=b+ad_0,\\
(aH+bF)D_{x_0}D_{x_1}D_y&=2aH^4+2(b-4ad)H^3F=\tfrac15(b+ad).
\end{align*}
The other triples do not add any extra conditions.
\end{proof}

The complete linear system $|F|$ defines a toric fibration
$f\colon \FF\to\PP^1$ whose fibre is the weighted projective space
$\PP(1,1,2,5)$. The singular locus of $\FF$ is the disjoint union of two torus
invariant rational curves, corresponding to the two isolated singularities of
$\PP(1,1,2,5)$.
These are the two sections
\begin{align}\label{formula!section}
\mathfrak{s}_2 &=D_{x_0} \cap D_{x_1} \cap D_z,&
\mathfrak{s}_5 &=D_{x_0} \cap D_{x_1} \cap D_y.
\end{align}
Indeed, in a neighbourhood of every point of $\mathfrak{s}_2$
resp.~$\mathfrak{s}_5$, $\FF$ is analytically isomorphic to the product of a
smooth $1$-dimensional disc with the corresponding singularity of
$\PP(1,1,2,5)$: a quotient singularity of type $\frac12(1,1,1)$
resp.~$\frac15(1,1,2)$.

In particular $\FF$ is  ${\mathbb Q}$-Gorenstein of index $\lcm(2,5)=10$.
Since $F$ and $10H$ are Cartier, we may consider the complete linear system
$|10(H-dF)|$.

\subsection{Gorenstein regular simple fibrations}
\begin{df}\label{def!Grsf}
 A \emph{Gorenstein regular simple fibration in $(1,2)$-surfaces} of type
$(d,d_0)$ is an element $X \in |10(H-dF)|$ on $\FF(d;d_0)$ with at worst
canonical singularities. We sometimes denote $X\subset\FF(d;d_0)$ by $X(d;d_0)$.
\end{df}

We abuse notation and write $f:=f|_X\colon X\to \PP^1$. Each fibre of $f$ is a hypersurface in a weighted projective $3$-space and
therefore $R^1f_*\hol_X=0$. By the Leray spectral sequence, this implies that
$q_1(X)=h^1(f_*\hol_{X})=h^1(\hol_{\PP^1})=0$. Therefore $X$ is
\emph{regular}.

The hypersurface $X$ is defined by a polynomial of the form
\[
\sum_{a_0+a_1+2a_2+5a_5 = 10}
c_{a_0,a_1,a_2}(t_0,t_1)x_0^{a_0}x_1^{a_1}y^{a_2}z^{a_5}
\]
where $c_{a_0,a_1,a_2}(t_0,t_1)$ is a homogeneous polynomial whose degree is
\begin{equation}\label{formula!degrees}
\deg c_{a_0,a_1,a_2} = -a_0(d-d_0)-a_1(d_0-2d) =
\frac{(a_0+a_1)d+(a_1-a_0)e}2.
\end{equation}

The choices we made in defining $\FF$ and $X$ imply that $\deg  c_{0,0,0}=\deg
c_{0,0,5} =0$. That is, the coefficients of $z^2$ and $y^5$ are constant.
After scaling $z$ we may assume that $c_{0,0,0}=1$ since otherwise $X$ would
contain $\mathfrak{s}_5$. The singular locus of $X$ would then be non-canonical,
a contradiction. Similarly we may scale $y$ to ensure $c_{0,0,5}=1$ since otherwise $X$ would
have $\mathfrak{s}_2$ as a non-canonical singular curve. Then by a coordinate change (completing the square) we make the
coefficients of all monomials $x_0^{a_0}x_1^{a_1}y^{a_2}z$ equal to zero. We are left with a
polynomial of the form
\begin{equation}\label{formula!z^2=}
z^2+y^5+\sum_{\substack{a_0+a_1+2a_2 = 10\\a_2\ne5}}
c_{a_0,a_1,a_2}(t_0,t_1)x_0^{a_0}x_1^{a_1}y^{a_2}
\end{equation}

We proved that $X \cap {\mathfrak s}_2=X \cap {\mathfrak
s}_5=\emptyset$. In particular  $X$ is contained in the smooth locus of $\FF$ and therefore it is {\it  Gorenstein}.

\begin{remark}\label{rem!involution}
Note that $X(d;d_0)$ has an involution obtained by changing the sign of the variable $z$, describing $X$ as double cover of $D_z$. The branch locus is the surface determined by the restriction of the polynomial \eqref{formula!z^2=} to $D_z$ and the index 2 rational curve $\s_2$ considered as a subscheme of $D_z$. Indeed, $D_z$ is a $\PP(1,1,2)$-bundle over $\PP^1$, see \S\ref{sec!quadric-bundle}.
\end{remark}

For fixed $d,d_0$ the varieties $X(d;d_0)$ form a unirational family. The next result determines when this family is not empty, and the type of singularities of the general element in it.  The proof is an exercise in Newton polytopes that we postpone to \S\ref{sec!proof-existence}.

\begin{prop}\label{prop!Existence}
Gorenstein regular simple fibrations in $(1,2)$-surfaces of type $(d,d_0)$  exist if and only if $d_0 \geq  \frac14 d$.
The singular locus of the general $X(d;d_0)$ is contained in the torus invariant section $\mathfrak{s}_0:= D_{x_1} \cap D_y \cap D_z$. More precisely
\begin{enumerate}
\item[(a)] $X$ is nonsingular iff $d \leq d_0 \leq \frac32 d$ or $d_0 = \frac78 d$;
\item[(b)] $X$ has $8d_0-7d$ terminal singularities iff $\frac78 d< d_0< d$;
\item[(c)] $X$ has canonical singularities along $\mathfrak{s}_0$ iff $\frac14d \leq d_0 < \frac78 d$. \qed
\end{enumerate}
\end{prop}
\begin{remark}
Since $\frac14 d \leq d_0 \leq \frac32 d$, we see that neither $d$ nor $d_0$
may be negative.
\end{remark}

By Proposition \ref{prop!omega_F} and the adjunction formula, the canonical
divisor class of $X(d;d_0)$ is
\begin{equation}\label{formula!K_X}
K_X=(K_\FF+X)|_{X}=(H-2F)|_{X}
\end{equation}

\begin{lem}\label{lemma!K is ample}
Suppose $X(d;d_0)$ satisfies the conditions of Proposition \ref{prop!Existence}. Then
\begin{enumerate}
\item $K_X$ is ample if and only if $\min(d,d_0) \geq 3$;
\item $K_X$ is nef if and only if $\min(d,d_0)\geq 2$.
\end{enumerate}
\end{lem}
\begin{proof}
We prove (1) since (2) is similar. By \eqref{formula!K_X}, $K_X=(H-2F)|_X$. By Prop.~\ref{prop!ample}, if $\min(d,d_0)\geq 3$, then $H-2F$ is ample on $\FF(d;d_0)$ and therefore its restriction to $X$ is ample too.

Conversely, consider
the curve $\Gamma:=X \cap D_{x_0} \cap D_{x_1}$ which is contained in $X$. Then $K_X\Gamma=d-2$ so $d\leq 2$ implies that $K_X$ is not ample.
Finally, if $d_0\leq2$ and $d\geq 3$ then $d_0 < \frac78 d$ and so $\s_0 \subset X$ by Prop.~\ref{prop!Existence}. Since $(H-2F)\s_0=d_0-2$ we are done.
\end{proof}

We now examine the canonical map of $X$. Let $\FF_e$ be the Hirzebruch surface with fibre $l$ and positive section $\delta$ with $\delta^2=e$. The class of the negative section is $\delta-el$.
\begin{prop}\label{prop!canonicalmodels}
Suppose $\min(d,d_0)\ge3$. Then the canonical map of $X(d,d_0)$ is a rational map
whose image is the embedding of the Hirzebruch surface $\FF_e$, $e=3d-2d_0$ via
the linear system $|(d_0-2)l+\delta|$.
\end{prop}
\begin{proof}
By \eqref{formula!K_X} and the vanishing of $H^1(\FF,\oo_{\FF}(-X+H-2F)) = H^1(\FF,K_{\FF})$, the canonical system of $X$ is spanned by the following $3d-2$ monomials:
\[t_0^{d_0-2}x_0,\dots, t_1^{d_0-2}x_0,\ t_0^{3d-d_0-2}x_1,\dots,t_1^{3d-d_0-2}x_1.\]
Thus $X$ is mapped to the image of the toric variety $D_y \cap D_z \cong \FF_{e}$ in $\PP^{3d-3}$. This is an embedding of $\FF_e$ because $d_0\ge3$.
\end{proof}
\begin{remark} The base locus of $|K_X|$ is the rational curve $\Gamma:=X\cap D_{x_0}\cap D_{x_1}$.
\end{remark}

Thus almost all (excluding a few degenerate cases with $d$, $d_0$ small, see Remark \ref{rem!nef-big}) Gorenstein
regular simple fibrations in $(1,2)$-surfaces are canonical threefolds with
canonical image a Hirzebruch surface. For each admissible pair $d,d_0$ we have a
unirational family of canonical threefolds that are all {\it on the Noether line}, as follows
\begin{teo}\label{thm!main?}
Gorenstein regular simple fibrations in $(1,2)$-surfaces of type $(d,d_0)$  are canonical 3-folds if and only if $\min (d,d_0) \geq 3$. In these cases
 \[
   p_g=3d-2,\ \ q_1=q_2=0,\ \ K_{X}^3=4d-6=\frac{4p_g-10}3.
 \]
Their canonical image is the Hirzebruch surface $\FF_e$, $e=3d-2d_0$. They form
a unirational family that is not empty if and only if $e \leq \frac52 d$.

The singular locus of the general $X(d;d_0)$ is contained in the torus invariant section $\mathfrak{s}_0:= D_{x_1} \cap D_y \cap D_z$ and more precisely it is
\begin{enumerate}
\item empty if $e \leq  d$ or $e = \frac54 d$;
\item $5d-4e$ terminal singular points if $d<e < \frac54d$;
\item $\mathfrak{s}_0$ if $\frac54d < e \leq \frac52d$.
\end{enumerate}
\end{teo}
\begin{proof}
Most of the statement follows by Lemma \ref{lemma!K is ample}, Proposition \ref{prop!canonicalmodels} and Proposition \ref{prop!Existence}, reformulating the inequalities in
Proposition \ref{prop!Existence} in terms of $e$ (instead of $d_0$) and $d$. It remains to prove the given formulas for the invariants.

We already showed that $p_g=3d-2$ and $q_1=0$. Since the Leray spectral
sequence of the direct image of $\hol_{X}$ degenerates at page $2$, we have
$h^2(\hol_{X})=h^0(R^2f_*\hol_{X})$. By Grothendieck duality
\[R^2f_*\hol_{X}
\cong f_*\hol_{X}(K_X+2F)^\vee\cong f_*\hol_{X}(H)^\vee\cong \hol_{\PP^1} (-d_0)
\oplus \hol_{\PP^1}(d_0-3d)\] and since $3d>d_0>0$ we get $q_2=0$.
Finally
$K_{X}^3=X(H-2F)^3=10(H^4-(d+6)H^3F)=4d-6$.
\end{proof}

\subsection{Simple fibrations with \texorpdfstring{$K_X$}{KX} nef but not ample}\label{rem!nef-big}
By Proposition \ref{prop!Existence}, there are a small number of $X(d;d_0)$ with $\min(d_0,d)=2$ which still have at worst canonical singularities. The complete list is $X(2;3)$ and $X(d;2)$ for $d=2,\dots,8$. In all of these cases, $K_X$ is nef and big (big because $K_X^3>0$) and the invariants of are the same as those of Theorem \ref{thm!main?}, so these also lie on the Noether line. Below we discuss these cases in more detail, first the case $d=2$ and then the cases $d\ge3$.

\begin{example}[see {\cite[Remark 2.3]{CH}}] The canonical image of $X(2;3)$ is $\FF_0$ i.e.~$\PP^1\times\PP^1$, and the canonical model is the complete intersection $X_{2,10}\subset\PP(1^4,2,5)$, where the quadric equation does not contain the variable of weight 2. We see that $X(2;3)\to X_{2,10}$ contracts the base curve $\Gamma=X\cap D_{x_0}\cap D_{x_1}$ of $|K_X|$ to a 3-fold ordinary double point at $(0,0,0,0,-1,1)$. The other small resolution gives a second simple fibration in $(1,2)$-surfaces, corresponding to the other ruling on $\FF_0$. The two fibrations are related by the Atiyah flop. The canonical model of $X(2;2)$ is still $X_{2,10}$, but now the rank of the quadric has dropped to three and $X$ has a curve of singularities.
\end{example}

\begin{example}\label{exa!X(d;2)}
For each $X(d;2)$ with $d=3,\dots,8$, the image of the canonical map is the cone $\bar\FF_e$ over a rational normal curve of degree $e=3d-4$. Indeed the canonical model of $X(d;2)$ is obtained by contracting the curve $\s_0$ to an isolated canonical singularity lying over the vertex of $\bar\FF_e$.

The varieties $X(2;2)$ and $X(7;2)$ appeared recently in the literature. More precisely, a  hypersurface in a weighted projective space birational to them is in  \cite[Table 10]{CJL}, respectively in line 7 and line 11. The other examples seem to be new. The variety $X(8;2)$ is a canonical $3$-fold with $p_g=22$ and $K^3=26$ with singular canonical image. This shows that the bound $p_g \geq 23$ in \cite[Theorem 1.2, (3)]{HZ} is optimal, a question left open there.
\end{example}

\subsection{Proof of Proposition \ref{prop!Existence}}\label{sec!proof-existence}

We assume throughout that $X$ is general.
If $d_0 \geq d$, by (\ref{formula!degrees}) all $c_{a_0,a_1,a_2}$ have nonnegative degree. Thus $X$ is a general element of a base point free linear system contained in the smooth part of $\FF$ and therefore $X$ is smooth by the classical Bertini Theorem.

From now on, we assume that $d_0<d$ and examine the Newton polytope of $X$. The base locus of $|X|$ is $\s_0$. Indeed, it follows from \eqref{formula!degrees} that $\deg c_{10,0,0}<0$ and $\deg c_{0,10,0}\ge0$. In particular any singularities of $X$ lie on $\s_0$. In fact, by \eqref{formula!degrees} we have $\deg c_{a_0,0,a_2}<0$ for all $a_0,a_2$. Thus the polynomial (\ref{formula!z^2=}) has the form
\[z^2 +y^5+x_1(c_{9,1,0}(t_0,t_1)x_0^9+g(t_0,t_1,x_0,x_1,y))
\]
where $g$ vanishes along $\mathfrak{s}_0$.

First suppose that $d_0\ge\frac78d$, or equivalently, $\deg c_{9,1,0}\ge0$. By generality, $c_{9,1,0}$ has distinct roots, and $X$ has $\deg c_{9,1,0}=8d_0-7d \geq 0$ isolated singular points on $\mathfrak{s}_0$ that are local analytically of the form $(tx_1 + z^2 + y^5 = 0)$. These are terminal singularities (cf.~\cite[Corollary 5.38]{KM}).
Notice that, if $d_0 = \frac78 d$ then by generality, $c_{9,1,0}$ is a nonzero constant, and $X$ is smooth.

Assume now that $d_0<\frac78d$. Then the polynomial (\ref{formula!z^2=}) has the form
\[z^2 +y^5+x_1(c_{8,2,0}x_0^8x_1+c_{7,1,1}x_0^7y+c_{7,3,0}x_0^7x_1^2+c_{6,2,1}x_0^6x_1y+c_{5,1,2}x_0^5y^2+
g)
\]
where $g$ vanishes at $\mathfrak{s}_0$ with multiplicity at least $3$. So $X$ is singular along $\mathfrak{s}_0$.

By \cite[\S1.14]{C3f}, if the nonisolated singularities are canonical then the general fibre $X_t$ of $X\to\PP^1$ has Du Val singularities and the special fibres have at worst elliptic singularities (dissident points). Conversely, if the general fibre has Du Val singularities then $X$ has cDV singularities there, so is canonical (see e.g.~\cite[\S5.3]{KM}). For the dissident points, we will show directly, that there is a crepant blowup $X'\to X$ which has cDV singularities \cite[\S3]{YPG}.

The following Lemma gives a necessary and sufficient condition for $X_t$ to have at worst Du Val singularities:

\begin{lem}\cite[\S 4.6, \S4.9]{YPG}\label{lem!C3f}
Let $0\in S\colon(F=0)\subset\mathbb{A}^3$ be an isolated hypersurface
singularity. Then $0\in S$ is Du Val if and only if
in any analytic coordinate system, $F$ has monomials of weight $<1$ with respect
 to each of the weights
$\frac12(1,1,0)$, $\frac13(1,1,1)$, $\frac14(2,1,1)$, $\frac16(3,2,1)$.
\end{lem}

We next prove that $d_0\geq \frac{d}4$. Let $\mathbf{x}=x_1/x_0$, $\mathbf{y}=y/x_0^2$
and $\mathbf{z}=z/x_0^5$ be local fibre coordinates near  the point $\s_0\cap X_t$. Considering $\mathbf{x},\mathbf{y},\mathbf{z}$ as an analytic coordinate system with weights  $\frac14(1,1,2)$, we see that Lemma \ref{lem!C3f} ensures there is a nonvanishing $c_{a_0,a_1,a_2}$ with $a_1+a_2 < 4$. Since $a_0+a_1+2a_2=10$, that is equivalent to $a_0-a_1>2$ and then by a parity argument to $a_0-a_1 \geq 4$.  Since $a_0+a_1\le 10$, it follows from \eqref{formula!degrees} that $4d_0-d=\frac{10d-4e}2 \geq \deg c_{a_0,a_1,a_2} \geq 0$.

Finally we prove that if $d_0\geq \frac{d}4$ then the general $X$ has canonical singularities.
To do this, we apply Lemma \ref{lem!C3f} with all permutations of the weights on the local fibre coordinates.
We note that for general $X_t$, the local equation always contains the  monomials $\mathbf{z}^2$, $\mathbf{y}^5$ and $\mathbf{x}^3$, the latter because $\deg c_{7,3,0}=4d_0-d\geq 0$. The reader can easily check that for all the prescribed weights, at least one of these three monomials has weight $<1$. Thus if $c_{7,3,0}$ does not vanish at $t$, then $X$ has cDV singularities there.

By generality, $c_{7,3,0}$ has $4d_0-d$ distinct zeros. Over each of these, $X$ possibly has a dissident point, locally given by at worst $\mathbf{z}^2+\mathbf{y}^5+t\mathbf{x}^3=0$. This is not cDV, but the relevant affine chart of the crepant blowup is given by
\[\mathbf{z}=t^5\mathbf{z}',\ \mathbf{y}=t^2\mathbf{y}',\ \mathbf{x}=t^3\mathbf{x}'.\]
The blown-up variety $X'$ is defined locally by $\mathbf{z}'^2+\mathbf{y}'^5+\mathbf{x}'^3=0$ which is then cDV. Hence the dissident points of $X$ are also canonical. $\qed$

\subsection{Mori Dream Spaces}\label{MDS}

In this section we prove that the general $X(d;d_0)$ is a Mori Dream Space when $d\leq d_0$.
Here by ``general" we mean that $X(d;d_0)$ is an element of a suitable dense open subset of the linear system $|10(H-dF)|$.

By definition \cite[Definition 3.3.4.1]{ADHL} a Mori Dream Space is an irreducible normal projective variety with finitely generated divisor class group and finitely generated Cox ring. The divisor class group $\Cl(\cdot)$ is  the group of linear equivalence classes of Weil divisors on the variety. In particular it coincides with the Picard group $\Pic(\cdot)$ when the variety is smooth.

The main point is proving
\begin{prop}
If $d\leq d_0$ and $X$ is general then the natural map
\[\Cl(\FF(d;d_0)) \rightarrow \Cl(X(d;d_0))
\]
is an isomorphism.
\end{prop}
\begin{proof}
Note that $ |10(H-dF)|$ is nef but not ample by Proposition \ref{prop!ample}. In particular we cannot apply directly \cite[Theorem 1]{GrotLef}.

We consider a desingularisation $\tilde{\FF} \rightarrow \FF$ of the singular locus, the curves $\mathfrak{s}_2$ and $\mathfrak{s}_5$, of $\FF$. Let $E$ be the exceptional locus.

The general $X$ is a smooth $3$-fold that does not intersect $\mathfrak{s}_2$ or $\mathfrak{s}_5$, so its pull-back is a divisor $\tilde{X}$ in $\tilde{\FF}$ mapped isomorphically to $X$. The divisor $\tilde{X}$ is  big since $\tilde{X}^4=X^4=10^4(H^4-4dH^3F)=10^3 d>0$.  By  the first lines of the proof of Proposition \ref{prop!Existence}, since we assumed $d\geq d_0$, the linear system $|10(H-dF)|$ is base point free, and therefore $|\tilde{X}|$ is base point free as well.

We factor the restriction map $\rho \colon \Pic(\tilde{\FF}) \rightarrow \Pic(\tilde{X})$ through  $\Pic(\tilde{\FF}\setminus E)$
as follows
\[\Pic(\tilde{\FF}) \xrightarrow{\rho_1}  \Pic(\tilde{\FF}\setminus E)  \xrightarrow{\rho_2} \Pic(\tilde{X})
\]

Following \cite[Section 1]{GrotLef}, we have isomorphisms
\begin{align*}\Cl(\FF) &\cong \Pic({\tilde{\FF}} \setminus E)&
\Cl(X) &\cong \Pic(X)=\Pic(\tilde{X})&
\end{align*}
so our claim is that $\rho_2$ is an isomorphism.

By a standard argument (detailed in \cite[Section 1]{GrotLef}) $\rho_1$ is surjective, with kernel isomorphic to the free abelian group generated by the classes of the irreducible divisorial components of $E$.

Since $\tilde{X}$ is big and base point free (and $\dim \tilde{\FF}=4\geq 3$) we can apply  the Grothendieck--Lefschetz Theorem for big linear systems \cite[Theorem 2]{GrotLef}.

Part a) of the G--L Theorem shows that the kernel of $\rho$ is generated by the classes of the irreducible divisors of $\tilde{\FF}$ contracted to a point by the map induced by the linear system $|\tilde{X}|$. They are exactly the divisors supported on $E$, since no irreducible Weil divisor of $\FF$ is contracted to a point by $|10(H-dF)|$. So $\ker \rho =\ker \rho_1$ which, since $\rho_1$ is surjective, implies that $\rho_2$ is injective.

Finally, since $\dim \tilde{\FF}=4$, part c) of the G--L Theorem shows  that in our situation, $\rho$ is surjective, and therefore  $\rho_2$ is surjective too.
\end{proof}

When the pull-back map $\Cl(\FF(d;d_0)) \rightarrow \Cl(X(d;d_0))$ is an isomorphism  \cite[Corollary 4.1.1.5]{ADHL} (see also \cite{AL}) can be applied giving directly
\begin{teo}\label{thm!MDS}
If $d \leq d_0$ and $X$ is general, defined by a polynomial $f$ as in \eqref{formula!z^2=} then the Cox ring of $X$ is
\[\CC[t_0,t_1,x_0,x_1,y,z]/f
\]
In particular $X$ is a Mori Dream Space.
\end{teo}
\begin{proof}
Using the notation of \cite{ADHL}, let $\bar X$ be the affine hypersurface $\{f=0\}$ in $\CC^6$ and let $\hat X = \bar X\smallsetminus \{t_0=t_1=0\}\cup\{x_0=x_1=y=z=0\}$ be the subset of $\bar X$ obtained by removing the irrelevant locus. The only relevant component of $\bar X \smallsetminus \hat X$ is $\{z^2+y^5=t_0=t_1=0\}$ which has codimension $2$ in $\bar X$. Hence the last assumption of \cite[Corollary 4.1.1.5]{ADHL} is fulfilled.
\end{proof}

\section{Deformations of threefolds on the Noether line}\label{sec!moduli}
In this section we study deformations of the canonical threefolds constructed in \S\ref{sec!Gorenstein}.
By Theorem \ref{thm!main?} we have canonical threefolds $X(d;d_0)$ on the Noether line for every $d,d_0$ with $d,d_0 \geq 3$, $0\leq e \leq \frac52 d$.
Since $p_g=3d-2$ is invariant under deformation, in the rest of this section we will consider $d\geq 3$ fixed.

The projection onto coordinates $(t_0,t_1;x_0,x_1)$ defines a rational map $\FF(d;d_0) \dashrightarrow \FF_e$ whose restriction to $X$ is the canonical map. The standard degeneration $\FF_{e} \leadsto \FF_{e+2}$ lifts easily to degenerations $\FF(d;d_0+1)\leadsto \FF(d;d_0)$.

We start by showing that the threefolds with minimal $e \leq 1$, that is $X\left( d; \left\lfloor \frac{3d}2 \right\rfloor \right)$ form a dense subset of an irreducible component of the moduli space.

We collect some preliminary vanishing results in the following:
\begin{lem}\label{lem!vanishing}
For every integer $n \geq 0$, we have
\begin{enumerate}
\item for all $q \neq 0$, $h^q(\hol_\FF(nF))=0$;
\item for all $n \leq d_0+1$, $h^1(\hol_\FF(H-nF))=0$;
\item if $e\leq d$, then $h^q(\hol_\FF(n(H-dF)))=0$ for all $q\ne0$;
\item if $e\leq d$, then $n'> 0$ implies $h^q(\hol_\FF(-n'(H-dF)-nF))=0$ for all $q\ne 4$.
\end{enumerate}
\end{lem}
\begin{proof}
(1) This follows directly from the Demazure Vanishing Theorem \cite[Thm 9.2.3]{CLS} since $F$ is nef.

(2) If $n \leq d_0+1$ then $n-1 \leq d_0 \leq 3d-d_0$. Thus $H^0(\hol_\FF(H-(n-1)F))$ contains multiples of both $x_0$ and $x_1$. Hence the restriction to a fibre $H^0(\hol_\FF(H-(n-1)F)) \rightarrow H^0(\hol_ {\PP(1,1,2,5)}(1))\cong \CC^2$ is surjective.
Suppose that $H^1(\hol_\FF(H-n_0F))$ vanishes for some $n_0\geq n$. Then the claim follows by recursively applying the cohomology exact sequence associated to the exact sequence
\begin{equation*}
0 \rightarrow \hol_\FF(H-nF) \rightarrow \hol_\FF(H-(n-1)F) \rightarrow \hol_ {\PP(1,1,2,5)}(1) \rightarrow 0,
\end{equation*}
Indeed, for $n\leq \min (d,d_0)$ we have $H-nF$ is nef by Proposition \ref{prop!ample} and then $h^1(\hol_\FF(H-nF))=0$ by the Demazure Vanishing Theorem.

(3) By Proposition \ref{prop!ample}, if $e\leq d$, then $H-dF$ is nef. The statement follows again by the Demazure Vanishing Theorem.

(4) This follows by Batyrev--Borisov vanishing \cite[Thm 9.2.7]{CLS}. Indeed, since $e\leq d$, the divisor $N:=n'(H-dF)+nF$ is a sum of nef divisors and therefore nef. We only need then to show that, for a  divisor $D=\sum a_\rho D_\rho$ in the class of $N$,  the polytope
\[P_D:=\{m \in M_\RR \mid \langle m,u_\rho \rangle \geq  -a_\rho\} \subset M_\RR \cong \RR^4
\]
has an internal point. We choose $D=\frac{n'}2 D_y+nD_{t_1}$. Recall that $M_\RR \subset \RR^6$  is the orthogonal of the two bottom rows of (\ref{formula!weightmatrix}) and choose $0 < \epsilon \ll1$. A direct computation shows that $\epsilon(d,d,1,1,-6,2)$ is an internal point of $P_D$.
\end{proof}
Now we can prove the announced result:
\begin{prop}\label{prop!mainstream}
The threefolds $X\!\left( d; \left\lfloor \frac{3d}2 \right\rfloor \right)$ form a dense open subset of an irreducible component of the moduli space.
\end{prop}

We need to prove that every small deformation of a smooth $X\left( d; \left\lfloor \frac{3d}2 \right\rfloor \right)$ is still an $X\left( d; \left\lfloor \frac{3d}2 \right\rfloor \right)$. Looking at the exact sequence defining the normal bundle of $X$ in $\FF$
\[
0 \rightarrow T_X \rightarrow T_{\FF|X} \rightarrow N_{X|\FF} \rightarrow 0
\]
we see that it suffices to prove $H^1(T_{\FF|X})=0$, since then the induced map $H^0 (N_{X|\FF}) \rightarrow H^1 (T_X)$ is surjective. So Proposition
\ref{prop!mainstream} is a consequence of the following:
\begin{lem}
If $d_0=\left\lfloor \frac{3d}2 \right\rfloor $ then
$H^1(T_{\FF|X})=0$.
\end{lem}
\begin{proof}
By the restriction exact sequence
\[
0\rightarrow T_\FF(-X) \rightarrow T_\FF \rightarrow T_{\FF|X} \rightarrow 0
\]
we need only prove that $H^1(T_\FF )$ and $H^2( T_\FF(-X) )$ vanish.

Consider the cohomology exact sequence associated to the dual of the Euler sequence (see \cite[Thm 8.1.6]{CLS})
\begin{equation}\label{sequence!EulerDual}
0 \rightarrow \hol^2_\FF \rightarrow \bigoplus_\rho  \hol_\FF (D_\rho) \rightarrow T_\FF \rightarrow 0.
\end{equation}
We use Lemma \ref{lem!vanishing}.
By part (1), $h^2(\hol_\FF)=h^1(\hol_\FF(D_{t_j}))=0$; by part (3) $h^1(\hol_\FF({D_{y}})=h^1(\hol_\FF(D_{z}))=0$;
by part (2) $h^1(\hol_\FF(D_{x_0}))=0$ and, since  $d_0=\left\lfloor \frac{3d}2 \right\rfloor \Rightarrow 3d-d_0\leq d_0+1$, also $h^1(\hol_\FF(D_{x_1}))=0$.
Then $h^1(T_\FF )=0$.

 Consider now the tensor product of the sequence (\ref{sequence!EulerDual}) by $\hol_\FF(-X)$. By Lemma \ref{lem!vanishing}, part (4), $h^3(\hol_\FF(-X))=h^2(\hol_\FF({D_{x_0}}-X)=h^2(\hol_\FF(D_{x_1}-X))=h^2(\hol_\FF({D_{y}}-X)=h^2(\hol_\FF(D_{z}-X))=0$;
moreover also $h^2(\hol_\FF(-X))=0$ and then by
\[
0\rightarrow \hol_\FF(-X)\rightarrow \hol_\FF(D_{t_j}-X) \rightarrow \hol_{\PP(1,1,2,5)}(-10)  \rightarrow 0
\]
since (\cite[1.4.1]{Dolgachev})  $h^2(\hol_{\PP(1,1,2,5)}(-10) )=0$ also $h^2(\hol_\FF(D_{t_j}-X))=0$.
Then $h^2(T_\FF(-X) )=0$.
\end{proof}

Now we try to lift the degenerations $\FF_{e-2} \leadsto\FF_{e}$  to degenerations $X(d;d_0+1)\leadsto X(d;d_0)$ by the argument of \cite[Remark 1.3]{CanonicalPencils}. We first construct a degeneration $\FF(d;d_0+1)\leadsto \FF(d;d_0)$. Assume then $d < \left\lfloor \frac{3d}{2} \right\rfloor$ ($e \geq 2$)  and let $\tilde\FF$ be the toric variety with weight matrix
\[\renewcommand{\arraystretch}{1.2}
\begin{pmatrix}
t_0 & t_1 & \tilde x_0 & x_0 & x_1 & y & z\\
1 & 1 & d-d_0 & d-d_0-1 & d_0-2d+1 & 0 & 0\\
0 & 0 &  1 &  1  &  1 &    2 &  5
\end{pmatrix}\]
and irrelevant ideal $(t_0,t_1)\cap(\tilde x_0,x_0,x_1,y,z)$.

Consider the family $\tilde\FF \times\Lambda \supset\F\to\Lambda$ defined by
\begin{equation}\label{ScrollarDeformation}
\lambda \tilde x_0 = t_0x_0-t_1^{e-1}x_1
\end{equation}
with parameter $\lambda\in\Lambda$ where $\Lambda$ is a disc around $0\in\CC$.

Set as usual $\F_\lambda$ for the fibre over $\lambda \in \Lambda$.
Then for $\lambda\neq 0$ the equation (\ref{ScrollarDeformation}) eliminates $\tilde x_0$ and thus the fibre $\F_\lambda$ is isomorphic to $\FF(d;d_0+1)$. On the contrary $\F_0 \cong \FF(d;d_0)$ with ``coordinates'' $t_0,t_1,\tilde x_0,\tilde x_1:=\frac{x_1}{t_0}=\frac{x_0}{t_1^{e-1}},y,z$.

We choose generators $\tilde{F}$ and $\tilde{H}$ of $\Cl(\tilde\FF)$ defined respectively by $t_0$ and  $t_0^{d_0+1}x_0$.
Notice that the restrictions of $\tilde{H}$ and $\tilde{F}$ to $\F_\lambda\cong \FF(d;*)$  give respectively the classes $H$ and $F$. This is obvious for $\tilde{F}$ and for $\lambda \neq 0$. For the restriction of $\tilde{H}$ to $\F_0 \cong \FF(d;d_0)$ it follows since  the divisors of $t_0^{d}\tilde x_0$ and $t_0^{d+1}x_0$ belong to the same class.

Then for every $\tilde X\in |10(\tilde{H}-d\tilde{F})|$ on $\tilde\FF$ the flat family
\[(\tilde X \times \Lambda) \cap\F=:\X\to\Lambda
\]
defines a degeneration $X(d;d_0+1)\leadsto X(d;d_0)$ if all fibres have canonical singularities.
This works very well when $d_0$ is big enough, i.e.~when $e$ is small enough.

\begin{prop}\label{prop!BoundMainStream}
If $e\leq d$, then every $X(d;d_0)$ lies in the closure of  the family of the $X(d;\left\lfloor \frac{3d}2 \right\rfloor)$.
\end{prop}
\begin{proof}
Arguing as in the previous section we may assume that $\tilde X$ is defined by a polynomial of the form
\begin{equation*}
z^2+\sum_{\tilde a_0+a_0+a_1+2a_2 = 10} \tilde c_{\tilde a_0,a_0,a_1,a_2}(t_0,t_1)(\tilde x_0)^{\tilde a_0}x_0^{a_0}x_1^{a_1}y^{a_2}
\end{equation*}
analogous to (\ref{formula!z^2=}). Intersecting with $\F_0\cong \FF(d;d_0)$, we can substitute $x_0=t_1^{e-1}\tilde x_1$ and $x_1=t_0\tilde x_1$ to get a polynomial of the form
$z^2+\sum_{\tilde a_0+\tilde a_1+2a_2 = 10} c_{\tilde a_0,\tilde a_1,a_2}(\tilde x_0)^{\tilde a_0}(\tilde x_1)^{\tilde a_1}y^{a_2}$
with
\[
c_{\tilde a_0,\tilde a_1,a_2}=\sum_{a_0+a_1= \tilde a_1} t_0^{a_1}t_1^{a_0(e-1)} \tilde c_{\tilde a_0,a_0,a_1,a_2}.
\]
Recall that, by (\ref{formula!degrees}), $\deg c_{\tilde a_0,\tilde a_1,a_2}=\tilde a_1\left(\frac{d+e}2 \right)+\tilde a_0\left(\frac{d-e}2 \right)$, so
\[
\deg c_{\tilde a_0,\tilde a_1,a_2}-\tilde a_1e=(\tilde a_0+\tilde a_1) \left( \tfrac{d-e}{2}\right).
\]
Then by the assumption $e\leq d$ it follows that $\deg c_{\tilde a_0,\tilde a_1,a_2} \geq \tilde a_1 e$. Since every homogeneous polynomial in $\CC[t_0,t_1]$ of degree $\ge \tilde a_1e$ belongs to the ideal $(t_0,t_1^{e-1})^{\tilde a_1}$, it follows that the polynomial of  $\X_0$ in $\FF(d;d_0)$ may assume every possible value of $(  \ref{formula!z^2=})$. Hence every $X(d;d_0)$ can be deformed to a $X(d;d_0+1)$.
\end{proof}

The condition $e \leq d$ is necessary in the previous proof to ensure that the scrollar deformations deform {\it every}  $X(d,d_0)$ ($d_0 \neq \lfloor \frac{3d}2 \rfloor$) to some  $X(d,d_0+1)$.  For $e>d$ the situation looks more tricky, and it seems that we have more components.

This however includes {\it almost all} smooth threefolds; by Proposition \ref{prop!Existence}, we only miss those with $5d=4e$, that is $7d=8d_0$. In fact, they belong to a different component, see the forthcoming Theorem \ref{thm!2comps}.

\section{On \texorpdfstring{$\PP(a_1,\ldots,a_r)$}{P(a1,...,ar)}-bundles}\label{sec!bundles}

We develop some foundations for weighted $\PP^r$-bundles over a nonsingular base
$B$, generalizing the work of Mullet (\cite{Mullet}). Such bundles
can be constructed by taking relative Proj of a sheaf $\SSS$ of graded $\oo_B$-algebras. We do
not assume that $\SSS$ is generated in degree 1.

\subsection{Weighted symmetric algebras}\label{subsec!WSA}
\begin{df}\label{def!weighted-symmetric-algebra}
Let $B$ be an algebraic variety, $a_i$ positive integers.
A \emph{weighted symmetric algebra} $\SSS$ on $B$ with weights
$(a_1,\ldots,a_n)$ is a sheaf of graded $\hol_B$-algebras $\SSS:=\bigoplus_{d
\geq 0}\SSS_d$ such that $\SSS_0 \cong \hol_B$ and $B$ is covered by open sets
$U$ with the property
\begin{equation}\label{local triviality}
\SSS_{|U}\cong \hol_{U}[x_1,\ldots,x_n]
\end{equation}
where $\hol_U[x_1,\ldots,x_n]$ is graded by $\deg x_i=a_i$.
We sometimes use the shorthand notation $a^r$ to denote $a$ repeated $r$ times.
\end{df}
\begin{example}
If $\E$ is a locally free sheaf on $B$ of rank $r$, then $\Sym(\E)$ is a
weighted symmetric algebra with weights $(1^r)$.
\end{example}
The inclusion $\SSS_1 \subset \SSS$ induces an injective morphism of sheaves of
algebras  $\Sym (\SSS_1) \rightarrow \SSS$ which is an isomorphism if and only
if $a_i=1$ for all $i$. Therefore the weighted symmetric algebras with weights
$(1^r)$ are exactly the usual symmetric algebras $\Sym(\E)$ where $\E$ is a
locally free sheaf of rank $r$ on $B$. Similarly all  weighted symmetric
algebras with weights $(a^r)$ are isomorphic to some $\Sym(\E)$ up to changing
the grading as follows:
\begin{example}
Let $a$ be a positive integer. We define
$\Sym^{(a)}(\E)=\bigoplus_{d\ge0}\Sym^{(a)}(\E)_d$ where
\[\Sym^{(a)}(\E)_{d} \cong
\begin{cases}
\Sym(\E)_k & \text{if }d=ka \\
0 & \text{otherwise}
\end{cases}
\]
The algebra structure is inherited from the natural isomorphism with
$\Sym(\E)$.
\end{example}

If we take two weighted symmetric algebras $\SSS$ and $\SSS'$ with respective
weights $(a_1,\ldots,a_m)$ and $(a'_1, \ldots, a'_n)$, then $\SSS
\otimes_{\hol_B} \SSS'$ has a natural structure of weighted symmetric algebra
with weights $(a_1,\ldots a_m,a'_1,\ldots,a'_n)$. This leads us to the
following definition:
\begin{df}\label{def!sym-a}
Choose positive integers $a_1<a_2<\cdots < a_n$ and  locally free sheaves
$\E_{a_1},\ldots,\E_{a_n}$ over $B$.
Then we define the \emph{associated free} weighted symmetric algebra
\[
\wSym_{a_1,\ldots,a_n} (\E_{a_1} , \cdots , \E_{a_n}):=(\Sym^{(a_1)}
\E_{a_1})\otimes_{\hol_B}\cdots \otimes_{\hol_B} (\Sym^{(a_n)} \E_{a_n})
\]
whose weights are $(a_1^{r_{a_1}},\dots,a_n^{r_{a_n}})$ where $r_{a_i}=\rank
\E_{a_i}$.
 \end{df}

 To consider more general weighted symmetric algebras we will need the following
 \begin{df}
 Let $\SSS$ be a weighted symmetric algebra. For a nonnegative integer $\tau$,
we define the \emph{truncated subalgebra} $\SSS[\tau]$ as the sheaf of
subalgebras locally generated by $1$ and $\{x_j \mid \deg x_j \leq \tau\}$ as an
$\oo_U$-algebra (see Def.~\ref{def!weighted-symmetric-algebra} for notation).
\end{df}

\begin{example}
 If $\SSS$ is a weighted symmetric algebra with weights $a_1<a_2<\ldots<a_n$
then $\SSS[0]=\hol_B$, $\SSS[\tau]=\SSS$ if and only if $\tau \geq a_n$.
\end{example}

\begin{example}
 If $\SSS$ is a weighted symmetric algebra with weights $(a^r)$ then
$\SSS[\tau]=\SSS$ if $\tau \geq a$, whereas $\SSS[\tau]=\SSS_0\cong \hol_B$ if
$\tau < a$.
\end{example}
\begin{example} Since we assumed that $a_i<a_{i+1}$ in Def.~\ref{def!sym-a}, we
have
\[
\wSym_{a_1,\ldots,a_n} (\E_{a_1} , \cdots ,
\E_{a_n})[a_i]=\wSym_{a_1,\ldots,a_{i}} (\E_{a_1} , \cdots , \E_{a_i}).
\]
\end{example}

More generally, if $\SSS$ is a weighted symmetric algebra with weights
$(a_1^{r_1}, \ldots, a_n^{r_n})$ where $a_1<\cdots <a_n$ then $\SSS[\tau]$ is a
weighted symmetric algebra with weights $(a_1^{r_1}, \ldots, a_t^{r_t})$, where
$a_t=\max \{ a_j\mid a_j \leq \tau \}$.

Truncation enables us to define an analogue of the sheaves $\E_{a_j}$ for any
weighted symmetric algebra.

\begin{df}\label{def!AssociatedSheaves}
Let $\SSS$ be a weighted symmetric algebra with weights $(a_1^{r_1}, \ldots,
a_n^{r_n})$ where $a_1<\cdots <a_n$. For every $1\le j\le n$, the
\emph{characteristic sheaf} of degree $a_j$ is the cokernel $\E_{a_j}(\SSS)$ of
the natural inclusion
\[\sigma_{a_j}\colon \SSS[a_j-1]_{a_j}\hookrightarrow \SSS_{a_j}.\]
Since $\SSS_{|U}\cong\oo_U[x_1,\dots,x_r]$, this is a locally free sheaf of rank $r_j$. We denote the projection maps by
$\epsilon_{a_j} \colon \SSS_{a_j} \rightarrow \E_{a_j}(\SSS)$.
\end{df}

\begin{remark} A weighted symmetric algebra with weights $(a_1^{r_1}, \ldots,
a_n^{r_n})$  is free (see Def.~\ref{def!sym-a}) of the form $\wSym_{a_1,\ldots,a_n} (\E_{a_1} , \cdots ,
\E_{a_n})$ if and only if all $\epsilon_{a_j}$ have a right inverse.
\end{remark}

The proof of the following Proposition is left as an exercise.
\begin{prop}\label{prop!construction}
Let $\SSS$ be a weighted symmetric algebra with weights $(a_1^{r_1}, \ldots,
a_n^{r_n})$,  with $a_1<\cdots <a_n$.
The natural map
\[
\SSS[a_n-1] \otimes_{\hol_B} \Sym^{(a_n)} \left( \SSS_{a_n} \right) \rightarrow
\SSS
\]
is surjective and its kernel is the ideal sheaf locally generated by the
elements of the form $u \otimes 1 - 1 \otimes \sigma_{a_n}(u)$.
\end{prop}

The maps $\sigma_{a_j}$ determine $\SSS$ recursively. Indeed, we can use the
above Proposition to construct every weighted symmetric algebra with weights
$(a_1^{r_1}, \ldots, a_n^{r_n})$ where $a_1<\cdots <a_n$ according to the
following algorithm:
\begin{description}
\item[Step 1] Set $\SSS[a_1-1]=\SSS[0]=\hol_B$, the symmetric algebra which is zero in degrees $>0$.
\item[Step 2] Given $\SSS[a_j-1]$, choose a locally free sheaf $\SSS_{a_j}$ and
an inclusion $\sigma_{a_j}\colon\SSS[a_j-1]_{a_j} \rightarrow \SSS_{a_j}$ with locally
free cokernel $\E_{a_j}$. Then by Prop.~\ref{prop!construction}, we define
$\SSS[a_{j}]=\SSS[a_{j+1}-1]$ to be the quotient of $\SSS[a_j-1]
\otimes_{\hol_B} \Sym^{(a_j)} \left( \SSS_{a_j} \right)$ by the ideal sheaf
locally generated  by the elements of the form $u \otimes 1 - 1 \otimes
\sigma_{a_n}(u)$.
\item[Step 3] Finally set $\SSS:=\SSS[a_n]$.
\end{description}

It is helpful to work out the specific case of weighted symmetric algebras
$\SSS$ with weights $(1,1,2,5)$ in detail.
The primary example to have in mind is $\wSym_{1,2,5}(\E_1, \E_2, \E_5)$, in
which case the maps $\epsilon_2$, $\epsilon_5$ below have a right inverse.

The characteristic sheaves of $\SSS$  are three vector bundles $\E_1$, $\E_2$
and $\E_5$ of respective ranks $2,1,1$.
Set $\SSS_1:=\E_1$ and $\SSS[1]=\Sym\SSS_1$. We get the short exact sequence
\[0\to(\Sym\SSS_1)_2 \xrightarrow{\sigma_2}\SSS_2
\xrightarrow{\epsilon_2}\E_2\to0\]
where $\sigma_2$ is locally the inclusion
$\hol_U[x_0,x_1]_2\to\hol_U[x_0,x_1,y]_2$.
In this case, $\SSS[2]_5$  can be written down explicitly as cokernel of the injective map
\[
\SSS_1 \otimes \SSS_2 \otimes \det \SSS_1 \rightarrow \SSS_1 \otimes \Sym^2
\SSS_2
\]
given by
\begin{equation*}\label{kernel of sigma_5}
x \otimes y \otimes (x' \wedge x'') \mapsto x' \otimes (y \sigma_2(xx'') )
-  x''\otimes (y \sigma_2(xx') ).
\end{equation*}

The map $\sigma_5$ is locally the inclusion
$\hol_U[x_0,x_1,y]_5\to\hol_U[x_0,x_1,y,z]_5$, giving the exact sequence
\[0 \rightarrow \SSS[2]_5 \xrightarrow{\sigma_5} \SSS_5 \xrightarrow{\epsilon_5}
\E_5 \rightarrow 0.\]
Since the highest weight is $5$, we have constructed $\SSS$.

 \subsection{Bundles in weighted projective spaces}
\begin{df}
Let $\SSS$ be a weighted symmetric algebra with weights $(a_1,\ldots,a_n)$. Then
$\FF:=\Proj_B(\SSS)$ is called a $\PP(a_1,\ldots,a_n)$-bundle over $B$.
\end{df}

By definition, $\FF$ comes with a natural projection $\pi \colon \FF \rightarrow
B$ whose fibres are all isomorphic to the weighted projective space
$\PP(a_1,\ldots,a_n)$. There are also sheaves $\hol_{\FF}(d)$
(\cite[(3.2.5.1)]{EGA2}) for all $d \in \ZZ$ whose restriction on each fibre is
isomorphic to the sheaf $\hol_{\PP(a_1,\ldots,a_n)}(d)$. For every coherent sheaf $\F$ on
$\FF$ we write as usual $\F(d)$ for $\F \otimes \hol_\FF(d)$.

\begin{remark}
By definition for all $d \geq 0$, $\pi_* \hol_\FF(d) \cong \SSS_d$, and for all
$d < 0$, $\pi_* \hol_\FF(d) =0$.
\end{remark}

\begin{remark}
If $L$ is a line bundle on $B$ then $\Proj_B(\SSS)\cong\Proj_B(\SSS\widehat\otimes L)$
where $\SSS\widehat\otimes L$ is the weighted symmetric algebra with $(\SSS\widehat\otimes
L)_d=\SSS_d\otimes L^d$.
\end{remark}

\begin{remark}\label{rem!toric}
 If $B=\PP^k$, $\SSS=\wSym_{a_1,\ldots,a_n} (\E_{a_1} , \cdots , \E_{a_n})$ and
all $\E_j$ split as sums of line bundles, then $\FF$ is the toric variety in
\cite[Construction 3.2]{Mullet}.
\end{remark}

 \begin{example}\label{example!1125toric}
Suppose $\E_1$ is locally free of rank $2$ and $\E_2$, $\E_5$ are line bundles
on $B=\PP^1$. Then $\SSS:=\wSym_{1,2,5}(\E_1, \E_2, \E_5)$ is a weighted
symmetric algebra with weights $(1,1,2,5)$.
Since every vector bundle on $\PP^1$ splits as a direct sum of line bundles,
we may write $\E_1=\hol(d_0)\oplus\hol(d_1)$, $\E_2=\hol(d_2)$,
$\E_5=\hol(d_5)$.

 The relative Proj of $\SSS$ over $B$ is naturally isomorphic to the toric
variety $\CC^6/\!/(\CC^*)^2$ with weight matrix
 \[
 \begin{pmatrix}
t_0 & t_1 & x_0 & x_1 & y & z\\
1 & 1 & -d_0 & -d_1 & -d_2 & -d_5\\
0 & 0 &  1 &      1 &    2 &  5
\end{pmatrix}
\]
and irrelevant ideal $(t_0,t_1)\cap(x_0,x_1,y,z)$.
If there exists $d\in \ZZ$ such that $d_1=3d-d_0$, $d_2=2d$, $d_5=5d$ then
${\Proj_B} (\SSS)$ is isomorphic to $\FF(d;d_0)$, the toric
variety with weight matrix (\ref{formula!weightmatrix}) of \S\ref{sec!Gorenstein};
indeed, the reader can check that the latter is isomorphic to
$\Proj_B(\SSS\otimes\oo_{\PP^1}(-d))$.
\end{example}

\begin{remark}
Not every $\PP^n$-bundle is relative Proj of a symmetric algebra. There is an obstruction which is a torsion element in $H^2(\hol_B^*)$. Examples are known where $B$ is a 2-dimensional complex torus \cite{EN}.
\end{remark}

\subsection{Relative dualising sheaf}

\begin{df}
We say that a $\PP(a_1,\dots,a_n)$-bundle over $B$ is \emph{well-formed} if the
fibre $\PP(a_1,\dots,a_n)$ is well-formed. In other words, if
$\hcf(a_1,\dots,a_n)=1$ and $\hcf(a_1,\dots,\widehat{a_i},\dots,a_n)=1$ for all
$i$.
\end{df}
Let $\FF$ be a well-formed $\PP(a_1,\dots,a_n)$-bundle. Then $\FF$ is singular
in codimension $\ge2$, and we denote by $j\colon W\to\FF$ the inclusion of the
nonsingular locus $W$ inside $\FF$. Recall that
$\omega_{\FF/B}=j_*\omega_{W/B}$.
\begin{prop}\label{prop!cotangent-PP}
Let $\FF=\Proj_B(\SSS)$ be a well-formed
$\PP(a_1^{r_1},\dots,a_n^{r_n})$-bundle, $a_1<a_2<\cdots < a_n$.
There is a sheaf $\V$ on $\FF$ and an exact sequence:
\begin{equation}\label{eq!euler}
0\to\Omega_{\FF/B}\to\V\to\hol_{\FF}\to0.
\end{equation}
This is the relative Euler sequence in the sense that its restriction to a fibre of $\pi\colon\FF\to B$
gives the Euler sequence for $\PP(a_1^{r_1},\dots,a_n^{r_n})$ (\cite[\S2]{Dolgachev}).

Then we have
\begin{enumerate}
\item There is an exact sequence
\[0\to\bigoplus_k\pi^*\SSS[a_k-1]_{a_k}\otimes\oo_{\FF}(-a_k)\to\bigoplus_k\pi^*\SSS_{a_k}\otimes\oo_{\FF}(-a_k)\to\V\to0.\]
\item The relative dualising sheaf of $\FF$ is
\[\textstyle{\omega_{{\FF}/B}\cong\pi^*\left( \bigotimes_k  \det \E_{a_k} \right)  \left( -\sum_k r_k a_k \right)},
\]
where $\E_{a_k}$ is the characteristic sheaf of $\SSS$ in degree $a_k$.
\end{enumerate}
\end{prop}

\begin{proof}
First consider the following natural map of $\SSS$-modules:
\[\tilde\varphi\colon\bigoplus_{k=1}^n\SSS_{a_k}\otimes\SSS(-a_k)\to\SSS\]
which is defined on each direct summand by multiplication of sections $s\otimes t\mapsto \deg s\cdot st$ and then extending by linearity.
If we choose local isomorphisms
\[\SSS_{a_k}|_U\cong\oo_U[x_1,\dots,x_r]_{a_k}\cong\bigoplus_{i=1}^{N}\oo_U\cdot \dd m_{k_i},\]
where $\dd m_{k_i}$, $i=1,\dots,N$ are the free generators corresponding to the monomials $m_{k_i}$ in $\SSS_{a_k}|_U$ of degree $a_{k}$, then $\tilde\varphi$ maps $\dd m_{k_i}$ to $a_k m_{k_i}$. The degree shift ensures that $\deg \dd m_{k_i}=\deg m_{k_i}=a_k$.

Let $\tilde\Omega_{\SSS/\oo_B}$ be the kernel of $\tilde\varphi$ and $\tilde\Omega_{\FF/B}$ its associated sheaf on $\FF$. Then the induced maps of associated sheaves on $\FF$ form a short exact sequence
\[0\to\tilde\Omega_{\FF/B}\to\bigoplus_{k=1}^n\pi^*\SSS_{a_k}\otimes\oo_{\FF}(-a_k)\to\oo_{\FF}\to0\]
where the exactness on the right follows because $\tilde\varphi$ is surjective in degrees $\ge1$.

Note that $\tilde\Omega_{\SSS/\oo_B}$ is a locally free $\SSS$-module and using the isomorphism $\SSS|_U\cong\oo_U[x_1,\dots,x_r]$ it has local generators
\[a_i \cdot x_i\dd x_j-a_j\cdot x_j\dd x_i \ \ \text{ and }\ \ \dd m_k-\sum_{i=1}^r\dfrac{\partial m_k}{\partial x_i}\dd x_i,\]
where here we write $a_i$ for the degree of $x_i$.
On the other hand, $\Omega_{\SSS/\oo_B}$, is locally generated as an $\SSS|_U$-module by the generators of the first type; that is, $a_i \cdot x_i\dd x_j-a_j\cdot x_j\dd x_i$.

Let $\K:=\bigoplus_{k=1}^n\SSS[a_k-1]_{a_k}\otimes\SSS(-a_k)$. We construct a map $\alpha\colon\K\to\tilde\Omega_{\SSS/\oo_B}$ whose cokernel is $\Omega_{\SSS/\oo_B}$. Locally, we define $\alpha_U\colon\K|_U\to\tilde\Omega_{\SSS/\oo_B}|_U$ by
\[m\otimes 1\mapsto\dd m-\sum_{i=1}^r\tfrac{\partial m}{\partial x_i}\dd x_i\]
where $m=m(x_1,\dots,x_r)$ is a section of $\SSS|_U[a_k-1]_{a_k}$. Next we show that $\alpha$ is well-defined. Suppose that $\SSS|_V\cong\oo_V[x_1',\dots,x_r']$. The transition function on $U\cap V$ is an isomorphism $v\colon\oo_{U\cap V}[x_1',\dots,x_r']\to\oo_{U\cap V}[x_1,\dots,x_r]$, $x_i=v_i(x'_1,\dots,x'_r)$. We denote the induced isomorphisms on $\K$ and $\tilde\Omega_{\SSS/\oo_B}$ by $v$ as well. Then
\[
\alpha_U(m(x))=\alpha_U(m(v(x')))=\dd m(v(x')))-\sum_{i=1}^r \dfrac{\partial m(v(x'))}{\partial v_i}\dd v_i(x')
\]
Now by the chain rule, we have
\begin{align*}
\left(\tfrac{\partial m(v(x'))}{\partial x_1'},\dots,\tfrac{\partial m(v(x'))}{\partial x_r'}\right)&=\left(\tfrac{\partial m(v)}{\partial v_1},\dots,\tfrac{\partial m(v)}{\partial v_r}\right)\cdot D_{x'}v(x')\ \ \text{and }\\
(\dd v_1,\dots,\dd v_r)^t&=D_{x'}v(x')\cdot(\dd x_1',\dots,\dd x_r')^t.
\end{align*}
Multiplying the first equation on the right by $D_{x'}v(x')^{-1}$ and combining them, we get
\[\sum_{i=1}^r \dfrac{\partial m(v(x'))}{\partial v_i}\dd v_i(x')=\left(\tfrac{\partial m(v(x'))}{\partial x_1'},\dots,\tfrac{\partial m(v(x'))}{\partial x_r'}\right)\cdot(\dd x_1',\dots,\dd x_r')^t.\]
It follows that
\[
\alpha_U(m(x))=\dd m(v(x'))-\sum_{i=1}^r \dfrac{\partial m(v(x'))}{\partial x'_i}\dd x_i'=\alpha_V(m(v(x')))
\]
and thus $\alpha$ is well-defined.
Hence we have a short exact sequence
\[0\to\mathcal{K}\xrightarrow{\alpha}\tilde\Omega_{\SSS/\oo_B}\to\Omega_{\SSS/\oo_B}\to0.\]

The two short exact sequences involving $\tilde\Omega_{\SSS/\oo_B}$ fit together as the middle row respectively first column of the following commutative diagram:
\[
\begin{tikzcd}
 & 0 \arrow{d} & 0 \arrow[d] \\
 & \K\arrow{d}{\alpha} \arrow[equals]{r} & \K \arrow[d]\\
    0 \arrow{r} & \widetilde\Omega_{\SSS/\oo_B} \rar \arrow{r}{\tilde\iota}\arrow{d}{\psi} & \bigoplus_{k=1}^n\SSS_{a_k}\otimes\SSS(-a_k) \arrow{d}{\beta}\arrow{r}{\tilde\varphi} & \SSS \arrow[equals]{d} \\
       0 \arrow{r} & \Omega_{\SSS/\oo_B} \rar\arrow{r}{\iota}\arrow{d} & \V_{\SSS} \arrow{r}{\varphi}\arrow{d} & \SSS \\
       & 0 & 0
    \end{tikzcd}
\]

The composition $\tilde\iota\circ\alpha\colon\K\to\bigoplus_{k=1}^n\SSS_{a_k}\otimes\SSS(-a_k)$ is injective. We call the cokernel $\V_{\SSS}$ and fill in the third row using a diagram chasing argument. Since $\tilde\varphi$ is surjective in degrees $\ge1$, $\tilde\varphi=\varphi\circ\beta$ and $\beta$ is surjective, it follows that $\varphi$ is surjective in degrees $\ge1$. Thus the sequence of sheaves of $\oo_{\FF}$-modules associated to the bottom row is the relative Euler sequence \eqref{eq!euler} on $\FF$.

We can now prove statements (1) and (2).

(1) is proved by taking the exact sequence of sheaves associated to the middle column of the above diagram.

(2) Since $\FF$ is well-formed, the singular locus of $\FF$ has codimension $\ge 2$. Hence it suffices to prove that the two sheaves are isomorphic on the nonsingular locus $W\subset\FF$. We restrict the relative Euler sequence \eqref{eq!euler} to $W$. Then this is an exact sequence of vector bundles and since $\omega_{W/B}=\det\Omega_{W/B}$, we deduce that $\omega_{W/B}=\det\V$. By part (1), we conclude that $\omega_{W/B}$ is the restriction of $\det\left(\bigoplus_{k}\pi^*\E_{a_k}(-a_k)\right)\cong\pi^*\left( \bigotimes_k  \det \E_{a_k} \right)(-\sum_k r_ka_k)$ to $W$.
\end{proof}

\begin{example} Let $\FF=\PP_B(\E)$ where $\E$ is a vector bundle of rank $r$
over $B$. Then $\varphi\colon \E\otimes\Sym\E\to(\Sym\E)(1)$ is the
canonical surjection (\cite[\S4.1]{EGA2}). The cotangent sequence reads
\[0\to\Omega_{\FF/B}\to\pi^*\E(-1)\to\oo_{\FF}\to0\]
and the relative dualising sheaf is $\omega_{\FF/B}=\pi^*(\det\E)(-r)$.
\end{example}
\begin{example}  With the same setup as Remark \ref{rem!toric}, $\FF=\Proj_B\SSS$ is toric, and
Proposition \ref{prop!cotangent-PP} specialises to the Euler sequence for toric $\PP(a_1^{r_1},\dots,a_n^{r_n})$-bundles (cf.~Prop.~\ref{prop!omega_F}).
\end{example}

\section{Simple fibrations in (1,2)-surfaces}\label{sec!simple}

\begin{df}\label{def!simple}
A \emph{simple fibration in $(1,2)$-surfaces} is a morphism
$\pi \colon X \rightarrow B$
between compact varieties of respective dimension $3$ and $1$
such that
\begin{enumerate}
\item $B$ is smooth;
\item $X$ has canonical singularities;
\item $K_X$ is $\pi$-ample;
\item for all $p \in B$, the canonical ring $R(X_p, K_{X_p}):=\bigoplus_d
H^0(X_p,K_{X_p})$ of the surface $X_p:=\pi^{*}(p)$
is generated by four elements of respective degree $1,1,2$ and $5$ and related
by a single equation of degree $10$.
\end{enumerate}
\end{df}

The fibres $X_p$ of $\pi\colon X\to B$ with at worst Du Val singularities are
$(1,2)$-surfaces.

\begin{remark} In applications we are interested in $X,B$ compact. The first
part of the forthcoming discussion can be generalized to $\pi\colon X\to B$
proper.
\end{remark}

\begin{remark}
Suppose that $X\to B$ is a fibration all of whose fibres $X_b$ are stable
Gorenstein surfaces with $p_g(X_b)=2$, $K_{X_b}^2=1$.
Then $X$ is simple by \cite[Thm 3.3, part 1]{FPR} and Thm \ref{thm!XcaninF}
below.

Non-simple fibrations $X \rightarrow B$ whose \emph{general} fibre is a
$(1,2)$-surface do exist; see \cite[Ex.~4.7]{FPR} or Ex.~\ref{eg!rank-2}.
\end{remark}

\begin{remark}
Theorem \ref{teo!splitting-epsilon-2} below proves that all the Gorenstein regular simple fibrations with $K_X^3=\tfrac{4p_g-10}3$
appear in Section \ref{sec!Gorenstein}. That is, under the above assumptions, $\epsilon_2\colon \SSS_2\to\E_2$ always has a right inverse.
\end{remark}

\begin{example}
Simple fibrations need not be Gorenstein nor stable. For example, consider
\[X\colon z^2=tf_{10}(t;x_0,x_1,y)\subset\PP_B(1,1,2,5)\]
where for simplicity, $B$ is a small disc with coordinate $t$ (it is not difficult to construct an example with $B$ compact). Then $X$ has nonreduced central fibre a
weighted projective space $\PP(1,1,2)$ with multiplicity $2$. Despite this, if $f|_{t=0}$ is general,
 then $X$ has only (nonisolated) canonical singularities.
\end{example}

\subsection{Relative canonical model}
Recall that the relative canonical sheaf
of $\pi\colon X\to B$ is
\[
\hol_X(K_{X/B}):=\hol_X(K_X - \pi^*K_B)
\]
and the relative canonical algebra of $\pi$ is the sheaf of $\hol_B$-algebras
\[
\R:=\bigoplus_{d\geq 0} \R_d:=\bigoplus_{d\ge 0} \pi_{*} \hol_X(dK_{X/B}).
\]
Since we assumed that $K_X$ is $\pi$-ample, $X$ and $\Proj_B\R$ are isomorphic.

\begin{teo}\label{thm!XcaninF}
Let $\pi \colon X\to B$ be a simple fibration in $(1,2)$-surfaces.
Then there is a weighted symmetric algebra $\SSS(X)$ with weights
$(1^2,2,5)$ such that $X$ is isomorphic to a hypersurface of relative degree
$10$ in the  $\PP(1,1,2,5)$-bundle $\FF(X):=\Proj_B(\SSS) \to
B$.
\end{teo}
\begin{proof}
Throughout this proof, we use implicitly the formula $\rank \R_n=h^0(S,K_S)$, where $S$ is a $(1,2)$-surface.
Hence $\rank \R_1=2$ and $\rank\R_n=3+\frac12n(n-1)$ for $n\ge2$.

We construct $\SSS(X)$ as follows:
First consider the weighted symmetric algebra $\Sym \R_1$ with weights $(1^2)$.
Since any fibre $S$ is mapped to $\PP^1$ by $|K_S|$, there are no
relations involving only variables of degree $1$. Hence the natural map $\Sym
\R_1 \to \R$ is injective and an isomorphism in degree $1$.

The multiplication map $\sigma_2 \colon \Sym(\R_1)_2 \rightarrow \R_2$ has
cokernel $\E_2$ which is locally free of rank $1$ because $S$ is mapped onto
$\PP(1,1,2)$ by $|2K_S|$. Hence we can construct (cf.~\S\ref{subsec!WSA}) a weighted symmetric algebra $\SSS'$
with weights $(1^2,2)$  with an injective morphism
$\SSS'\hookrightarrow\R$, which is an isomorphism in degrees $1$, $2$, $3$ and
$4$.

The cokernel $\E_5$ of the inclusion $\SSS_5'\hookrightarrow\R_5$ is locally free of
rank $1$, so we get a weighted symmetric algebra $\SSS \supset \SSS'$ with
weights $(1^2,2,5)$ such that $\SSS' \cong \SSS[2]$. There is a morphism
$\SSS \rightarrow \R$ that is an isomorphism in degrees $\le9$ and thereafter
surjective, so inducing an inclusion
\[
X\cong \Proj_B \R \subset \FF:=\Proj_B\SSS
\]
 of $X$ as divisor in a $\PP(1,1,2,5)$-bundle over $B$. The relative degree of $X$ is
 then $10$, the degree of the single equation defining its general fibre.
\end{proof}

\begin{remark}\label{rem!hypersurface-can}
Conversely, any divisor $X$ of relative degree $10$ in a $\PP(1^2,2,5)$-bundle
over a smooth curve $B$ and with at worst canonical singularities is a simple
fibration in $(1,2)$-surfaces. Thus from now on, we assume that a simple
fibration is a hypersurface of relative degree $10$ in a $\PP(1,1,2,5)$-bundle
with $\omega_{X/B}=\oo_{\FF}(1)_{|X}$.
\end{remark}

\subsection{\texorpdfstring{$X$}{X} as double cover of a \texorpdfstring{$\PP(1,1,2)$}{P(1,1,2)}-bundle}\label{sec!quadric-bundle}
We first describe the singular locus of $\FF$ as in
\S\ref{sec!Gorenstein}.

\begin{df}
The singular locus of a $\PP(1,1,2,5)$-bundle over $B$ is the disjoint union
of two sections ${\mathfrak s}_2$ and ${\mathfrak s}_5$, where ${\mathfrak s}_k$
has Gorenstein index $k$.
\end{df}
Since $X$ has at worst canonical singularities, we get some constraints on the
intersections $X \cap \s_k$.

\begin{prop}\label{prop!index 5} Let $\pi\colon X\to B$ be a simple fibration in
$(1,2)$-surfaces and suppose that $X\subset\FF(X)$ where $\FF(X)$ is the $\PP(1,1,2,5)$-bundle constructed
in Thm \ref{thm!XcaninF}. Then
\begin{enumerate}
\item $X \cap \s_5=\emptyset$;
\item $\s_2 \not\subset X$.
\end{enumerate}
\end{prop}
\begin{proof}
\emph{(1)} Suppose $p\in X\cap {\mathfrak s}_5$. Then, in a neighbourhood of
$p$, $\FF(X)$ has a singular point that is a quotient singularity of type
$\frac15(1,1,2,0)$. Since $X$ is a Cartier divisor, $P$ is a
noncanonical singular point of $X$, at best a $\frac15(1,1,2)$ point, which is
a contradiction (see also \cite[Rmk 4.6]{FPR}).

\emph{(2)} Suppose $\s_2\subset X$. For a general point $p$ in $B$, there is a local analytic neighbourhood $p\in V\subset B$ such that the equation of $X$ has the form $z^2=q(x_0,x_1)y^4+\dots$, where $q$ is (at best) a relative quadratic form over $V$. Thus in a neighbourhood of $\s_2$, $X$ looks like $V\times\{(z^2=q)\subset\frac12(1,1,1)\}$. This is at best a curve
of singularities $V\times\frac14(1,1)$ \cite{hacking-barcelona}, which is not canonical, a contradiction.
\end{proof}

We now show that $X$ is a double cover of a $\PP(1,1,2)$-bundle.
\begin{df}\label{def!double-cover}
We define the truncated subalgebra ${\mathcal Q}(X):=\SSS(X)[2]$ and let $g
\colon \FF(X) \dashrightarrow \QQ(X):=\Proj_B{\mathcal Q}(X)$ be the
natural map corresponding to the inclusion $\mathcal{Q}(X)\subset\SSS(X)$.
\end{df}
In the toric case of \S\ref{sec!Gorenstein} or Ex.~\ref{example!1125toric}, $\QQ(X)$ is naturally isomorphic to the torus invariant divisor $D_z$.

\begin{lem}
The restriction $g_{|X}\colon X\to \QQ(X)$ is a finite morphism
of degree 2. The double-cover involution on $X$ lifts to $\SSS$ (and $\mathcal R$) in such a
way that the invariant part of $\mathcal R$ is $\mathcal Q$.
\end{lem}
\begin{proof}
The indeterminacy locus of $g$ is $\s_5$, so by Prop.~\ref{prop!index
5}(1), the restriction $g_{|X} \colon X \to {\QQ(X)}$ is a finite
morphism of degree $2$. The involution on $X$ which swaps the two sheets of
this covering can be lifted to $\SSS$.

Indeed, on an open subset $U\subset B$, we have $\SSS_{|U}\cong
\hol_U[x_0,x_1,y,z]$ with $\deg x_j=1$, $\deg y=2$, $\deg z=5$.
Prop.~\ref{prop!index 5}(1), implies that the coefficient of $z^2$ in the
equation of $X$ never vanishes; completing the square we may assume that the equation has
the form $z^2=f(x_0,x_1,y)$. Then the involution may be lifted to
$\SSS_{|U}$ as the involution fixing $x_0,x_1,y$ and mapping $z \mapsto
-z$. These glue in the obvious way to give an involution on $\SSS$ and a splitting
into invariant and anti-invariant parts: $\SSS=\SSS^+ \oplus \SSS^-$.
By construction, the involution preserves $X$ so we get a splitting $\R=\R^+
\oplus \R^-$ and clearly $\R^+={\mathcal Q}$.
\end{proof}
\begin{remark}
Notice the analogy with the splitting of the
relative canonical algebra of a genus $2$ fibration induced by the hyperelliptic
involution of the fibres \cite[Lem.~4.3]{pencils}.
\end{remark}

\begin{example}\label{eg!rank-2}
If we allow the quadric cone to degenerate to a
quadric of rank two over a finite number of points of $B$, then we
obtain a fibration in $(1,2)$-surfaces that is not simple.
For example, consider the complete intersection
\[X\colon (x_0x_1=ty_0,\ z^2=f_{10}(t;x_0,x_1,y_0,y_1))\subset\PP_B(1,1,2,2,5)\]
where for simplicity, $B$ is a small disc with coordinate $t$. When $t$ is invertible, the fibre $X_t$ is just a hypersurface of degree $10$ in $\PP(1,1,2,5)$ because $y_0$ is eliminated using $\frac1t x_0x_1$. On the other hand, when $t$ is not invertible, the fibre $X_0$ is a reducible surface with two components $(x_i=0)$. In fact $X_0$ consists of two singular K3 surfaces glued along a line (see \cite[Ex.~4.7]{FPR}).
\end{example}

We describe $\R^-$ as a $\mathcal Q$-module.

\begin{prop}\label{prop!splitting} The map $\epsilon_5 \colon \SSS_5(X) \rightarrow \E_5$ has a right inverse. Moreover
\[
\R^- \cong {\mathcal Q}(-5) \otimes \E_5.
\]
\end{prop}
\begin{proof} For all $d\le4$ we get $\SSS_d^+={\mathcal Q}_d$ and thus $\SSS_d^-=0$.
In degree $5$, $\SSS_5^+={\mathcal Q}_5=\ker \epsilon_5$ so that $(\epsilon_5)_{|\SSS_5^-}$
is an isomorphism, whose
inverse is a right inverse for $\epsilon_5$. Hence $\R_5^- \cong \SSS_5^- \cong
\E_5$.

Now as locally free $\hol_B$-modules, $\R_d^+$ and $\R_d^-$ are generated
by the monomials $x_0^ax_1^by^c$ with $a+b+2c=d$
(resp.~$x_0^ax_1^by^cz$ with $a+b+2c=d-5$). Thus all multiplication maps
\[
{\R}^+_d \otimes \R_5^- \rightarrow {\R}^-_{d+5}
\]
are isomorphisms, completing the proof.
\end{proof}

Now we describe the ``equation" of $X$ in $\FF=\Proj_B(\SSS)\xrightarrow{\pi}B$.
For any line bundle $\LLL$ on $B$, there are natural isomorphisms $H^0(\FF,\hol_{\FF}(d)\otimes \pi_\FF^* \LLL^{-1}) \cong \Hom_{\hol_B}(\LLL, \SSS_d)$. So $X$ is defined by a map $
\LLL \hookrightarrow \SSS_{10}$ for a suitable line bundle $\LLL$.
The line bundle can be determined precisely
\begin{cor}\label{cor!Xmissess5}
The hypersurface $X\subset\FF$ is defined by an injective homomorphism
$
\E_5^2 \hookrightarrow \SSS^+_{10}.
$
\end{cor}
\begin{proof}
 Since the involution of $\FF(X)$ preserves $X$ the image of $\LLL$ is
 contained in the invariant part $\SSS^+_{10}$ (or in $\SSS^-_{10}$, but this
would contradict Prop.~\ref{prop!index 5}(1)). By Prop.~\ref{prop!splitting},
$\E_5 \cong \R_5^- \cong \SSS_5^-$. Then $\SSS_{10}^+$ splits as ${\mathcal
Q}_{10} \oplus \E_5^2$. This corresponds locally to separating the polynomials
in $x_0,x_1,y$ from those involving $z^2$.

Consider the induced projection $\SSS_{10}^+ \to \E_5^2$.
Then Prop.~\ref{prop!index 5}(1) says that the composition of maps
$
\LLL \to \SSS_{10}^+ \to \E_5^2
$
is surjective. Therefore, since both $\LLL$ and $\E_5$ are line bundles, $\LLL \cong \E_5^2$.
\end{proof}

We can now translate the conditions of Prop.~\ref{prop!index 5} coming
from the canonical singularities of $X$ in terms of the characteristic sheaves
$\E_1$, $\E_2$ and $\E_5$.

\begin{cor}\label{cor: L5}\
\begin{enumerate}
\item $\E_5 \cong (\det \E_1) \otimes \E_2$.
\item $h^0(\E_2^3 \otimes (\det \E_1)^{-2}) \neq 0$.
\end{enumerate}
\end{cor}
\begin{remark}
Note that (1) together with Prop.~\ref{prop!splitting} shows
that $\R$ is determined as a ${\mathcal Q}$-module by ${\mathcal Q}$ itself:
\[\R \cong {\mathcal Q} \oplus \left( {\mathcal Q}(-5) \otimes (\det \E_1)
\otimes \E_2 \right).
\]
\end{remark}
\begin{proof}
\emph{(1)} By Prop.~\ref{prop!cotangent-PP},
\[
\omega_{\FF/B}=\hol_{\FF}(-9) \otimes \pi_\FF^*\det(\E_1 \oplus \E_2 \oplus
\E_5)
\]
and since $X \in |\hol_\FF(10) \otimes \pi^* \E_5^{-2}|$, the
adjunction formula gives
\[
\omega_{X/B}=\hol_{X}(1) \otimes \pi^*\det(\E_1 \oplus \E_2 \oplus
\E_5^{-1}).
\]
Finally, by Rem.~\ref{rem!hypersurface-can}, we have $\omega_{X/B}\cong \hol_{X}(1)$,
so the
thesis follows immediately.

\emph{(2)} This proof is inspired by \cite[Definition 2.4 and Proposition 2.5]{CanonicalPencils}).
The map $\Sym^5(\epsilon_2) \colon \Sym^5 (\SSS_2)  \to \E_2^5$
induced by $\epsilon_2 \colon \SSS_2 \to \E_2$ factors
through ${\mathcal Q}_{10}$ giving a map $\alpha \colon {\mathcal Q}_{10}\to
\E_2^5$ whose kernel consists of those elements of ${\mathcal Q_{10}}$ vanishing
along $\s_2$.

By Cor.~\ref{cor!Xmissess5}, $X$ is defined by a map $\E_5^2 \rightarrow \SSS_{10}^+\cong {\mathcal Q}_{10} \oplus \E_5^2$, where locally, the factor $\E_5^2$ gives the multiples of $z^2$ and therefore $\E_5^2$ is in the ideal sheaf of $\s_2$. Hence $\s_2 \not\subset X$ if and only if the composition of the first component of this map $\E_5^2 \to {\mathcal Q}_{10}$ with $\alpha\colon{\mathcal Q}_{10}\to\E_2^5$ is not the zero map. Thus $\Hom_{\hol_B}(\E_5^2,\E_2^5)\neq 0$. Substituting $\E_5 \cong (\det \E_1) \otimes \E_2$, we obtain the result.
\end{proof}

The corollary suggests the following definition
\begin{df}
Let $\pi\colon X\to B$ be a simple fibration in  $(1,2)$-surfaces. Then
\[
N(X):=3\deg \E_2 -2 \deg \E_1=3\chi(\E_2)-2\chi(\E_1)+\chi(\hol_B) \geq 0
\]
\end{df}
Geometrically $N$ is the expected number of $\frac12(1,1,1)$ singularities on
$X$. In fact, by the proof of Cor.~\ref{cor: L5}(2), if the divisor in
$|\E_2^3 \otimes (\det \E_1)^{-2}|$ corresponding to the homomorphism in that proof
is reduced, then $X$ intersects $\s_2$ in $N$ quasismooth points of $X$, of
type $\frac12(1,1,1)$.

\subsection{The invariants  of a simple fibration in \texorpdfstring{$(1,2)$}{(1,2)}-surfaces}
To compute the invariants of a simple fibration in $(1,2)$-surfaces we need the following lemma.

\begin{lem}\label{lem: hiRd} Let $\pi \colon X \rightarrow B$ be a simple fibration in $(1,2)$-surfaces  and suppose $\LLL$ is any line bundle on $B$. Then for $i=0,1$ we have
\[h^i(\hol_{X}(dK_{X/B})\otimes \pi^*\LLL)=h^i(\R_d \otimes \LLL)
 \text{ for all } d \geq 1,\]
and for $i=2,3$ we have
\[h^i(\hol_{X}(dK_{X/B})\otimes \pi^*\LLL)=
\left\{\begin{array}{ll}
h^{i-2}(\LLL) & \text{ for } d=1\\
0 & \text{ for } d \geq 2.
\end{array}\right.\]
\end{lem}
\begin{proof}

Since the fibres are hypersurfaces in weighted projective space, we have
\begin{align*}
\pi_* \hol_{X}& \cong \hol_B;&
R^1\pi_*\hol_{X}(dK_{X/B}) =0& \text{ for all }d.
\end{align*}
Thus in combination with the projection formula
\[
R^i \pi_* (\hol_{X}(dK_{X/B}) \otimes \pi^*\LLL) \cong \left( R^i \pi_* \hol_{X}(dK_{X/B}) \right) \otimes \LLL,
\] we see that the Leray spectral sequence of the direct image of $\hol_{X}(dK_{X/B}) \otimes \pi^*\LLL$ degenerates at page $2$ for each $d$ and $\LLL$.

Whence for $i=0,1$ and any $d\ge 1$, we have
\[h^i(\hol_{X}(dK_{X/B})\otimes \pi^*\LLL)=
h^i(\pi_*\hol_{X}(dK_{X/B})\otimes\LLL)=h^i(\R_d \otimes \LLL).\]
When $i=2,3$, we get
\[h^i(\hol_{X}(dK_{X/B})\otimes \pi_\FF^*\LLL)=
h^{i-2}(R^2\pi_*\hol_{X}(dK_{X/B})\otimes\LLL).\]
If $d\geq2$ then $R^2\pi_* \hol_{X}(dK_{X/B})=0$ by the base change theorem, because the fibres are canonically polarised. Thus $h^i(\hol_X(dK_{X/B})\otimes \pi^*\LLL)=0$ for $d\ge2$. When $d=1$, $R^2\pi_*\hol_X(K_{X/B})\cong(\pi_*\hol_X)^{\vee}$ by Grothendieck duality, so $h^i(\hol_X(K_{X/B})\otimes \pi^*\LLL)=h^{i-2}(\pi_*\hol_X\otimes\LLL)=h^{i-2}(\LLL)$.
\end{proof}

Now we compute the birational invariants of $X$.

\begin{prop}\label{prop!invariants}
Let $\pi \colon X \rightarrow B$ be a simple projective fibration in $(1,2)$-surfaces.
Then
\begin{align*}
p_g(X)&=h^0(\E_1 \otimes \omega_B),& q_1(X)&=g(B)=:b,\\
 q_2(X)&=h^1(\E_1 \otimes \omega_B)\leq 2,& \chi(\omega_X)&=\chi(\E_1)-5\chi(\hol_B).
\end{align*}
\end{prop}
\begin{proof}
The first three equalities follow from Lem.~\ref{lem: hiRd} with
$\LLL=\omega_B$.
For the last, note that
$\chi(\omega_X)=p_g-q_2+q_1-1=\chi(\E_1\otimes\omega_B)-\chi(\hol_B)$.
Then by the Riemann--Roch Theorem for curves, $\chi(\E_1\otimes\omega_B)=\chi(\E_1)+2\deg(\omega_B)$, and the result follows.

The inequality $q_2 \leq \rank(\E_1)= 2$ follows then by the semipositivity of $\E_1$ (\cite[Thm III]{Viehweg}).
\end{proof}
Then we compute the top selfintersection of the canonical divisor of $X$.
\begin{prop}\label{prop!L2}
Let $\pi \colon X \rightarrow B$ be a simple fibration in (1,2)-surfaces.
Then
\[
K^3_{X} =\frac43 \chi(\omega_X)-2\chi(\hol_B)+\frac{N}6 =\frac43 (p_g-q_2)+\frac{10}3(q_1-1)+\frac{N}6.
\]
\end{prop}

\begin{proof}
By Lem.~\ref{lem: hiRd},  $\chi(\omega_X^2)=\chi(\R_2\otimes\omega_B^2)$;
twisting the exact sequence $
0\to \Sym^2(\R_1)\to \R_2 \to \E_2 \to 0$ by $\omega_B^2$,
we get
\begin{align*}
\chi(\omega_X^2)&=\chi(\R_2\otimes\omega_B^2)\\
&=\chi(S^2(\R_1\otimes\omega_B))+\chi(\E_2\otimes\omega_B^2)\\
&=3\chi(\R_1\otimes\omega_B)-3\chi(\hol_B)+\chi(\E_2)+\deg\omega_B^2.
\end{align*}
By Proposition \ref{prop!invariants}, this last line is equivalent to
\[\chi(\omega_X^2)=-3\chi(\hol_X)+\chi(\E_2)-4\chi(\hol_B).\]
On the other hand, the Riemann--Roch formula \cite[Cor.~10.3]{YPG} gives
\[
\chi(\omega_X^2)=\frac12 K_{X}^3-3\chi(\hol_X)+\frac N4.
\]
Combining these two expressions to eliminate $\chi(\omega_X^2)$ and simplifying gives
$K^3_{X} =2\chi(\E_2)-8\chi(\hol_B)-\frac{N}2$. Finally, substituting $\chi(\E_2)=
\frac13 \left( N + 2 \chi(\E_1)-\chi(\hol_B) \right)
$ and then $\chi(\E_1)=\chi(\omega_X)+5\chi(\hol_B)$ we obtain the result.
\end{proof}

As a corollary we get a Noether type inequality for simple fibrations in $(1,2)$-surfaces.
\begin{cor}\label{cor!Noether}
Let $\pi \colon X \rightarrow B$ be a simple fibration in $(1,2)$-surfaces.
Then $K^3_{X}\geq \frac13{(4(p_g-q_2)-10(1-q_1))}$ with equality holding if and only if $X$ is Gorenstein.
\end{cor}

We conclude this section with the following Theorem,
which shows that the results of section \ref{sec!Gorenstein} are complete.
\begin{teo}\label{teo!splitting-epsilon-2}
Let $\pi\colon X\to B$ be a Gorenstein regular simple fibration with $K_X^3=\frac{4p_g-10}3$. Then $X$ appears
in Section \ref{sec!Gorenstein}
\end{teo}
\begin{proof}
Since $X$ is regular, we know that $B=\PP^1$. Moreover $X$ is Gorenstein so $N=q_2=0$. By definition of $N$ and Corollary \ref{cor: L5}, there is a $d$ such that $\det \E_1=\oo_{\PP^1}(3d)$, $\E_2=\oo_{\PP^1}(2d)y$, $\E_5=\oo_{\PP^1}(5d)z$. Moreover, there is a unique $d_0\le 3d-d_0$ such that $\E_1=\oo_{\PP^1}(d_0)x_0\oplus\oo_{\PP^1}(3d-d_0)x_1$.

Now consider the short exact sequence
\[0\to \Sym^2\E_1\xrightarrow{\sigma_2} \SSS_2\xrightarrow{\epsilon_2}\E_2\to0.\]
The main point of the proof is to show that $\epsilon_2$ has a right inverse.

We have $\Sym^2\E_1=\oo(2d_0)x_0^2\oplus\oo(3d)x_0x_1\oplus\oo(6d-2d_0)x_1^2$ and $d,d_0\ge0$ by Fujita semipositivity.
In this case,
\begin{equation}\label{eq!ext}
\Ext_{\oo_{\PP^1}}^1(\E_2,\Sym^2\E_1)\cong H^1(\Sym^2\E_1\otimes\E_2^\vee)\cong
H^1(\oo(2(d_0-d))),
\end{equation}
because $3d\ge2d$ and $6d-2d_0\ge2d$.
Standard cohomological arguments allow us to deduce the existence of an inverse when $d_0\ge d$, so we assume that $d_0\le d-1$.

We complete the proof arguing by contradiction. We show that if the extension class in \eqref{eq!ext} corresponding to the above short exact sequence is non-trivial, then $\Hom_{\oo_{\PP^1}}(\E_5^2,\mathcal{Q}_{10})$ is zero. Arguing as in the proof of \ref{cor: L5}(2), this implies that $\s_2\subset X$ which is a contradiction.

Motivated by this, we define $\mathcal{I}=x_1\mathcal{Q}$, the sheaf of ideals locally principally generated by $\oo(3d-d_0)x_1$. Then define $\mathcal{T}=\mathcal{Q}/\mathcal{I}$ and note that $\Proj_{\PP^1}(\mathcal{T})$ is the divisor $(x_1=0)$ in the $\PP(1,1,2)$-bundle $\QQ(X)$ over $\PP^1$.

Replacing $\E_1$ with $\mathcal{Q}_1$ and $\SSS_2$ with $\mathcal{Q}_2$ in the above short exact sequence and quotienting by $\mathcal{I}$, the multiplication maps in $\mathcal{T}$ give the following exact sequence
\begin{equation}\label{eq!splitting}
0\to\Sym^2\mathcal{T}_1\cong\oo(2d_0)\to\mathcal{T}_2\xrightarrow{\tilde\epsilon_2}\E_2\to0
\end{equation}
and $\tilde\epsilon_2$ has a right inverse if and only if $\epsilon_2$ has, because of \eqref{eq!ext}.

Since $\mathcal{T}$ is generated in degree 2, and $\mathcal{T}_1$ has rank 1, the multiplication map $\Sym^{k}\mathcal{T}_2\to \mathcal{T}_{2k}$ is an isomorphism. Note that $\mathcal{T}_2$ is a direct sum of two line bundles on $\PP^1$. Thus if \eqref{eq!splitting} does not split, then the maximal degree in $\mathcal{T}_2$ is $<2d$ and hence the maximal degree in $\mathcal{T}_{2k}$ is $<2kd$.

Since $\mathcal{T}$ is a quotient of $\mathcal{Q}$, we get a surjective map $\mathcal{Q}_{10}\to\mathcal{T}_{10}$. Thus all summands of $\mathcal{Q}_{10}$ are line bundles of degree $<10d$. Since $\E_5^2\cong\oo(10d)$, it follows that $\Hom(\E_5^2,\mathcal{Q}_{10})$ is zero. Hence $\epsilon_2$ has a right inverse.

We already know that $\epsilon_5$ has a right inverse by Proposition \ref{prop!splitting}. Thus $\SSS(X)\cong\wSym_{1,2,5}(\E_1,\E_2,\E_5)$ and hence $\FF$ is a toric variety as in Example \ref{example!1125toric}. This finishes the proof.
\end{proof}

\section{More on threefolds on the Noether line}\label{sec!more-noether}

\subsection{Kobayashi's construction}

We relate our simple fibrations with the threefolds on the Noether line constructed by Kobayashi \cite{Kobayashi}, and
generalised by Chen--Hu \cite{CH}.

\begin{prop}\label{prop!RecognizeCH}
The smooth threefolds in \cite[Thm 1.1]{CH} are exactly those of Theorem
 \ref{thm!main?} part (1) with $e \leq d$.
\end{prop}
\begin{proof}
In \cite[Thm 1.1]{CH}, the authors generalise Kobayashi's construction to exhibit threefolds $Y(a,e)$ on the Noether line
with canonical image $\FF_{e}$ and $p_g(Y)=6a-3e-2$ for all pairs of integers
$a,e$ with $a\geq e\geq 0$ excepting $(a,e)=(2,2)$, $(1,1)$, $(1,0)$, $(0,0)$.

We show that $Y(a,e)$ is the same as our $X(d;d_0)$ with $d=2a-e$, $d_0=2d-a$.

Recall that there is a double cover $X(d;d_0)\to D_z$, where $D_z$ is a bundle in quadric cones over $\PP^1$ (see Rmk \ref{rem!involution}). We perform
a weighted blowup $D'_z\to D_z$ of the index 2 section $\s_2=D_{x_0}\cap D_{x_1}\cap D_z$ and a corresponding blowup $X'\to X$ of
the preimage of $\s_2$ in $X$, to obtain the following diagram:
\[
\begin{tikzcd}
     X' \rar \arrow{r}{}\arrow{d}{} & D'_z \arrow{d}{}\arrow{r} & \FF_e\arrow{d} \\
       X(d;d_0)\rar\arrow{r}{} & D_{z}\arrow{r} & \PP^1
    \end{tikzcd}
\]
Recall from \S\ref{sec!Gorenstein} that $\delta$ is the positive section and $l$ is the fibre on $\FF_e$. We claim that $D'_z$ is the following $\PP^1$-bundle over $\FF_e$:
\[D_z':=
\PP_{\FF_e}\left(\oo_{\FF_e}\oplus\oo_{\FF_e}(2\delta+(d+e)l)\right).\]
Indeed, in coordinates, the blowup $D_z'\to D_z$ is
given by
\[t_i\mapsto t_i,\ x_0\mapsto ux_0',\ x_1\mapsto ux_1',\ y\mapsto y\]
where $u$ is the section defining the exceptional divisor $E$. Note that $s=t_0^{2(d_0-d)}x_0^2$ is a section of
$\oo_{\FF_e}(2\delta+(d+e)l)$ because $2(d_0-d)=d+e$. The rational function $s/y$ on $D_z$ pulls back to $su^2/y$ on $D_z'$. Since $y$, $u^2$ are the fibre coordinates of $D_z'\to\FF_e$, we get
the claimed $\PP^1$-bundle over $\FF_e$.

Moreover,
 $X'\to D'_z$ is a double cover with branch locus $B+E$ where $B\in|\oo_{D_z'}(5)|$ is the strict
transform of the branch divisor of $X\to D_z$ and $E$ is the exceptional divisor of
$D_z'\to D_z$.
This is the Kobayashi--Chen--Hu construction with $d+e=2a$. The
condition $a\geq e$ is equivalent to the condition $e\leq d$.

The short list of exclusions $(a,e)=(2,2),\dots$ mentioned above are just the $X(d;d_0)$ which violate $\min(d,d_0)\geq 3$, i.e.~those with $K_X$ nonample.
\end{proof}

We thus have more smooth examples than \cite{CH}, namely the general
Gorenstein regular simple fibration in $(1,2)$-surfaces with $e=\frac54d$, that
is $8d_0=7d$.
The simplest possible example is $X(8;7)$:
\begin{example} \label{example!newone}
Choose $d=8$, $d_0=7$, so $e=10$.
The polynomial
\[
z^2+y^5+x_0^9x_1+(t^{90}_0+t^{90}_1)x_1^{10}
\]
defines a smooth 3-fold $X(8;7)\subset \FF(8;7)$.
By Theorem \ref{thm!main?} this is a canonical 3-fold with $p_g=22$ and
$K_X^3=26$, not belonging to the examples in  \cite[Thm 1.1]{CH} since in
this case $e=10$, $a=9$. Therefore it contradicts \cite[Prop. 4.6.(b)]{CH} and
consequently the last assertion in \cite[Thm 1.3]{CH}.
\end{example}

\begin{remark}[The discriminant can be disconnected]
If $e=\frac54d$ then we can still blow up the curve $\s_2$ in $D_z$ to
obtain a construction in the style of Kobayashi as explained above.
Thus $X'\to D'_z$ is a double cover branched over $E$ and the surface
$B$ defined by \[u^{10}(\underbrace{x_0'^9x_1'+\dots+c_{0,10,0}x_1'^{10}}_{a_0})+
u^8y(\underbrace{c_{7,1,1}x_0'^7x_1'+\dots}_{a_1})+\dots+y^5.\]
The blowup resolves the base locus of $|K_X|$, and the projection $X'\to\FF_e$ onto
coordinates $t_0,t_1,x_0',x_1'$ is a genus 2
fibration over $\FF_e$ with fibre $(z^2=\sum_i a_i(p)u^{10-2i}y^i)\subset\PP(1_u,2_y,5_z)$
 where $p$ is
a point of $\FF_e$ and $a_i(t_0,t_1,x'_0,x'_1)$ are the coefficients as in the above displayed formula.

The discriminant $\Delta\subset\FF_e$ of the genus two fibration $X'\to\FF_e$ is reducible, because all $a_i$ are divisible by $x_1'$. Moreover, the two components of $\Delta$ are disjoint because the monomial $x_0'^9x_1'$ appears in $a_0$ with constant nonzero coefficient.
\end{remark}

\subsection{A second component of the moduli space}
\begin{teo}\label{thm!2comps}
For every $p_g\geq 7$ of the form $3d-2$ let $\N_{p_g}^0$ be the
subset of the moduli space of canonical threefolds with geometric genus $p_g$ and $K^3=\frac43p_g-\frac{10}3$
 given by smooth simple fibrations in $(1,2)$-surfaces. Then  $\N_{p_g}^0$ has
 \begin{itemize}
 \item[] one connected component if $d$ is not divisible by $8$
 \item[] two connected components if $d$ is divisible by $8$.
 \end{itemize}
All these components are unirational.

One component is formed by those 3-folds with canonical image $\FF_e$, $0 \leq e \leq d$. This is an open subset of the moduli space of canonical $3$-folds.

When $d$ is divisible by $8$ there is a second component of the moduli space of canonical $3$-folds including smooth $3$-folds whose canonical image is $\FF_{\frac54 d}$.
The intersection of the closures of the components in the moduli space of canonical $3$-folds is not empty.

In particular the moduli space of canonical $3$-folds with given $p_g=3d-2$, $K^3=\frac43p_g-\frac{10}3$ is reducible when $d$ is divisible by $8$.
\end{teo}

\begin{proof}
By Propositions \ref{prop!Existence},  \ref{prop!mainstream}, \ref{prop!BoundMainStream} and \ref{cor!Noether} all
smooth simple fibrations with $e\neq \frac54d$ are Gorenstein regular of the form $X(d;d_0)$ with $d_0\geq d$. Moreover, they belong to the same irreducible component of the moduli space of canonical $3$-folds, whose general element is a $X\left(d;\left\lfloor \frac32d \right\rfloor\right)$.

Now assume $d$ divisible by $8$ and choose $d_0=\frac78 d$ so that $e=\frac54d$.
A general scrollar deformation $\X$ as in \S\ref{sec!moduli} gives a  degeneration $X(d;d_0+1)\leadsto X(d;d_0)$ with singular central fibre $\X_0$. Indeed, in the notation of the proofs of Propositions \ref{prop!Existence} and \ref{prop!BoundMainStream} in this case we get first of all
\begin{itemize}
\item $\deg c_{10,0,0} <0 \Rightarrow \mathfrak{s}_0\subset X(d;d_0)$
\item $\deg c_{9,1,0} =0$ and $c_{9,1,0} \in  (t_0,t_1)^{e-1} \Rightarrow c_{9,1,0}=0$ that implies $\mathfrak{s}_0\subset \Sing X(d;d_0)$ and then the degeneration  $\X_0$ can not be smooth.
\end{itemize}
 However for general $\X$, $\X_0$ is canonical. Indeed $\deg c_{7,3,0} =2e$. Then following the argument of the proof of Proposition \ref{prop!BoundMainStream} we may get any  $c_{7,3,0}  \in (t_0,t_1^{e-1})^3$ of degree $2e$: these are all the multiples of $t_0$, and in particular we may get  $c_{7,3,0} $ with distinct roots, that is the  condition we used in Proposition \ref{prop!Existence} to ensure that the general element has canonical singularities. In particular, near $t_0=0$, $\X_0$ looks like $(z^2+y^5+t_0x_1^3=0)$ which is canonical by \S\ref{sec!proof-existence}.

We now show that there is no degeneration $X(d;d_0)\leadsto X\left( d; \frac78 d \right)$ with $d_0 \ge d$ and all fibres nonsingular. The argument is inspired by Horikawa \cite[Lemma 7.3]{hor1}, although it is more complicated to set up in our situation. Suppose by contradiction, that $\X\to\Lambda$ is such a degeneration. The relative canonical linear system $|K_{\X/\Lambda}|$ gives a rational map $\X/\Lambda\to\F/\Lambda$ where $\F$ is a degeneration of surfaces $\FF_{3d-2d_0}\leadsto\FF_{\frac54d}$. If $3d-2d_0\ne 0$, then the Hirzebruch surface $\FF_{3d-2d_0}$ admits a unique fibration to $\PP^1$. On the other hand, if $3d-2d_0=0$, then $d\ge3$ and by \S\ref{prop!canonicalmodels}, one of the two fibrations $\FF_0\to\PP^1$ is distinguished by the canonical linear system of $X$. Thus each fibre of $\F/\Lambda$ has a unique distinguished fibration  to $\PP^1$ and hence $\X/\Lambda$ admits a map to $\PP^1_{\Lambda}=\PP^1\times\Lambda$ factoring through $\F/\Lambda$. Moreover, this map induces the fibration in $(1,2)$-surfaces
$\X_\lambda\to\PP^1$ on each fibre.

 Now, the relative bicanonical linear system $|2K_{\X/\PP^1_\Lambda}|$ endows $\X$ with a double cover structure of the quadric cone bundle
 $\mathcal{Q}\to\PP^1_{\Lambda}$. This is the relative version of the double cover $X\to\QQ(X)$ on each fibre as defined in \ref{def!double-cover}. The branch locus consists of a divisor $\mathcal{B}\subset\mathcal{Q}$ and the special section $\s_2\colon\PP^1_{\Lambda}\to \mathcal{Q}/\Lambda$ corresponding to the vertex on each fibre of $\mathcal{Q}/\PP^1_{\Lambda}$. In particular, we have a distinguished element $y$ which cuts out a divisor in
 $\mathcal{Q}$ which is isomorphic to $\F$. Thus the family $(\F,\mathcal{B}|_{y=0})$ is a degeneration of pairs
 \[(\FF_{3d-2d_0},B)\leadsto(\FF_{\frac54d},B_0)\]
where the general $B$ is irreducible, but the central $B_0$ is disconnected.  This is impossible since, as observed by Horikawa \cite[Lemma 7.3, p.~382]{hor1}, if $t$ is sufficiently close to $0$, then $\mathcal{B}_t$ must be disconnected.
\end{proof}

\section{Threefolds with \texorpdfstring{$K_X$}{KX} big but not nef}\label{sec!nef-big}
In this section we analyse those $X(d;d_0)$ with $\min(d,d_0)=0,1$ and at worst canonical singularities. Firstly, by Proposition \ref{prop!Existence} we have $0\le\frac14d\le d_0 \le\frac32d$. Hence if $\min(d,d_0)=0$ then $d=d_0=0$ and $X(0;0)$ is a product $\PP^1\times(S_{10}\subset\PP(1,1,2,5))$. Secondly, if $\min(d,d_0)=1$ then $d_0=1$ and there are four possibilities, the first of which is $X(1;1)$ which has Kodaira dimension $0$. The other three are more interesting to us:

\begin{prop}\label{prop!K-flipping}
Consider $X=X(d;1)$ with $d=2,3,4$. Then $X$ has canonical singularities along $\s_0$ and $K_X$ is big but not nef. After flipping the negative curve $\s_0$, we get a quasismooth variety $X^+(d;1)$ in $\FF^+(d;1)$ with $K_{X^+}$ nef and big. The invariants of $X^+$ are listed in Table \ref{table!flipping}.
\end{prop}

\begin{table}[ht]
\renewcommand{\arraystretch}{1.2}
\[\begin{array}{cccc}
X(d;1) & p_g & K_{X^+}^3 & \text{Singularities of }X^+ \\
\hline
X(2;1) & 4 & \frac94 & 2\times\frac12(1,1,1),\ \frac14(1,3,3) \\
X(3;1) & 7 & \frac{85}{14} & \frac12(1,1,1),\ \frac17(3,4,6) \\
X(4;1) & 10 & \frac{301}{30} & \frac12(1,1,1),\ \frac13(1,2,2),\ \frac15(1,4,4)
\end{array}\]
\caption{Threefolds with $K_X$ not nef}\label{table!flipping}
\end{table}

Since $X(2;1)$ has a model as a hypersurface in weighted projective space (see Remark \ref{rem!3-div-hypersurface}), it can be found using the methods of \cite{BKZ}, \cite{BK16}. The $3$-folds $X(3;1)$ and $X(4;1)$ are in \cite[Table 10]{CJL}, respectively in lines $8$ and $10$.
\begin{proof}
Since $d_0=1$ and $d\ge 2$, we know that $X$ is singular along $\s_0$ by Proposition \ref{prop!Existence}. Moreover, by Lemma \ref{lemma!K is ample} and its proof, we have $K_X\cdot\s_0=(-2F+H)\cdot \s_0=d_0-2=-1<0$, hence $K_X$ is not nef.

We determine a minimal model for $X$ by applying the toric minimal model program to $\FF(d;1)$ (see \cite[\S15]{CLS}). The ray spanned by the class of the curve $\s_0$ is extremal in $\NE(\FF)$ and there is a birational map $\FF\dashrightarrow\FF^+$ which flips $\s_0$ to a weighted projective plane $S^+$.
The flipped variety $\FF^+$ is toric with the same weight matrix as $\FF$, but the irrelevant ideal is changed to $(t_0,t_1,x_0)\cap(x_1,y,z)$. The nef cone of $\FF^+$ is $\RR_+(H-F)+\RR_+(H-dF)$, hence $H-2F$ is (at least) nef on $\FF^+$.

The birational transform $X^+$ is defined by the same element of $|{10(H-dF)}|$ as $X$ was, but we consider $X^+$ as a subvariety of $\FF^+$. By the above discussion, $K_{X^+}=(H-2F)_{X^+}$ is nef.

The rest of the proof is a case by case computation, showing that $X^+$ is quasismooth, determining the quotient singularities of $X^+$ and the invariants $p_g$ and $K^3$ (see the following example for $X^+(2;1)$).
\end{proof}

\begin{example}
Consider the toric variety $\FF^+(2;1)$ with weight matrix
$\left(\begin{smallmatrix}
1 & 1 & 1 & -3 & 0 & 0 \\
0 & 0 &  1 &  1 &  2 &  5
\end{smallmatrix}\right)$ and irrelevant ideal $(t_0,t_1,x_0)\cap(x_1,y,z)$. Let $X^+(2;1)$ be a general element of the linear system $|{10(H-2F)}|$. After the usual coordinate changes (see \S\ref{sec!Gorenstein}), the equation of $X^+$ can be written as
\[z^2+y^5+x_0^3x_1y^3+x_0^6x_1^2y=x_1g(t_0,t_1,x_0,x_1,y)\]
where $g$ is contained in the ideal $(t_0,t_1)$ and for simplicity, we set all coefficients to be 1. More precisely, a Newton polygon computation shows that $x_1g(1,1,1,x_1,y)$ does not contain any monomials
$x_1^\alpha y^\beta$ with $\beta<5-2\alpha$ and thus $X(2;1)$ has a curve of $D_6$ singularities along $\s_0$.

Taking the irrelevant ideal into account, we see that the base locus of $|{10(H-2F)}|$ in $\FF^+(2;1)$ is the single point $P\colon(t_0=t_1=y=z=0)$. Moreover $X^+$ is quasismooth at $P$ because the equation contains the monomial $x_0^6x_1^2y$, so the affine cone over $X^+$ is nonsingular at $P$.

The flipped locus on $\FF^+$ is $S^+$ defined by $t_0=t_1=0$ which implies that $x_0\ne 0$ because of the irrelevant ideal. Hence we can rescale $x_0$ to eliminate one of the $\CC^*$-actions on $\FF^+$. Row operations on the weight matrix show that the remaining $\CC^*$-action reduces to
\[\begin{pmatrix}
t_0 & t_1 & x_1 & y & z \\
-1 & -1 & 4 & 2 & 5
\end{pmatrix}\]
thus $S^+\cong\PP(4,2,5)$ and the flipped curve $\s_0^+=X^+\cap S^+$ is defined by $z^2+y^5+x_1y^3+x_1^2y=0$ in $S^+$.

Next we determine the quotient singularities of $X^+$. Using the row-reduced weight matrix, we may identify the orbifold charts on $\FF^+$ covering $\s_0^+$; they are $U_{x_0,x_1}\cong\frac14(3,3,2,1)$, $U_{x_0,y}\cong\frac12(1,1,2,1)$ and $U_{x_0,z}\cong\frac15(4,4,4,2)$. Note that $U_{x_0,x_1}$ and $U_{x_0,y}$ cover a curve $\Gamma\cong\PP(4,2)$ of $\frac12(1,1,1)$ singularities containing a dissident $\frac14(3,3,1)$ point. Since $X^+\cap\Gamma$ is defined by $y^5+x_1y^3+x_1^2y=0$, we see that $X^+$ contains two $\frac12(1,1,1)$ points and the $\frac14(3,3,1)$ point. The other chart $U_{x_0,z}$ has an isolated index 5 singularity which is not contained in $X^+$ because of the monomial $z^2$.

Thus the basket of singularities of $X^+$ is $\{2\times\frac12(1,1,1),\frac14(3,3,1)\}$. Moreover,
$\chi(\hol_X)=1-0+0-4=-3$ and $P_2(X)=11$. Next we apply the orbifold Riemann--Roch formula \cite{YPG} for $\chi(2K_X)$:
\[
\frac12K_{X^+}^3=P_2(X)+3\chi(\hol_X)-\sum_{Q\in\B}\frac{b(r-b)}{2r}
\]
where $Q\cong\frac1r(1,-1,b)$, to get
\[K_{X^+}^3=2\cdot\left(11+3\cdot(-3)-\textstyle{2\cdot\frac14-\frac38}\right)=\textstyle{\frac{9}{4}}.\]
\end{example}

\begin{remark}[Hypersurface model of $X^+(2;1)$]\label{rem!3-div-hypersurface}
Recall that $K_{X^+(2;1)}$ is nef but not ample. In this case, there is a model of $X^+(2;1)$ as
a hypersurface in weighted projective space:
\[X_{30}\subset\PP(1,1,4,6,15).\]
This has $3$-divisible canonical class, an additional Gorenstein canonical singularity $\frac13(1,1,1)$ on the line $\PP(6,15)$. Let $a_0,a_1,b,c,d$ denote the coordinates on $\PP(1,1,4,6,15)$. The contraction $X^+(2;1)\to X_{30}$ is given by
\[(a_0,a_1,b,c,d)=(\sqrt[3]{x_1}t_0,\sqrt[3]{x_1}t_1,\sqrt[3]{x_1}x_0,y,z),\]
which is the crepant resolution of the $\frac13(1,1,1)$ point. The pencil $|\hol(1)|$ are surfaces with $p_g=2$, $K^2=\frac43$, $2\times A_1$, $A_3$ and a $\frac13(1,1)$ singularity, the minimal resolution being a $(1,2)$-surface.
\end{remark}

\subsection{Non-terminal flips}
The birational map $X\dashrightarrow X^+$ is a non-terminal flip, because $X$ is singular along $\s_0$. One approach to describing this map is would be to resolve the singularities along $\s_0$ and then run the MMP to get a minimal model. See Figure
\ref{fig: schematic} for a schematic picture of this for $X(2;1)$.

\begin{figure}[ht]
\begin{center} \begin{tikzpicture}[thick, every node/.style = {font = \small}]

  \draw (3.25,0.2) to (3.25,1.2);
  \draw (2.5,1) to (3.25,1.2);
  \draw (2.5,0) to (3.25,0.2);

  \draw (2.5,0) to (2.5,1);
  \draw (1.75,1.2) to (2.5,1);
  \draw (1.75,0.2) to (2.5,0);

  \draw (1.75,0.2) to (1.75,1.2);
  \draw (1,1) to (1.75,1.2);
  \draw (1,0) to node[below ]{resolution of $D_6\times\PP^1$} (1.75,0.2);
  \draw (1,0) to (1,1);

  \draw (0.7,1.05) to (1,1);
  \draw (0,0.2) to (1,0);
  \draw(0,0.2) to (0,0.4);

  \draw (0.7,1.05) to (0.2,1.6);
  \draw[dashed] (0.7,0.05) to (0.7,1.1);
  \draw (0.7,0.05) to (0.2,0.6);
  \draw (0.2,0.6) to (0.2,1.6);

  \draw (0.15,1.25) to (-0.2,1.6);
  \draw (0.3,0.15) to (-0.2,0.6);
  \draw[dashed] (0.3,0.2) to (0.3,0.5);
  \draw (-0.2,0.6) to (-0.2,1.6);

  \draw[->] (1.6,1.6) node[above]{\tiny{contract ruling}} to (0.9,1.1);
  \draw[dashed,->] (3.5,0.6) to (4.5,0.6);
  \draw (5,0.1) to[bend left] node[below]{$\s_0^+$}(7,0.1);

 \draw (5.2,0.2) to (4.8,1.7);
  \draw (5.2,0.2) to (4.5,1.1);
\draw(4.5,1.1) to (4.8,1.7);

 \draw (6,0.4) to (5.6,1.85);
  \draw (6,0.4) to (5.3,1.25);
\draw(5.3,1.25) to (5.6,1.85);

  \draw (7.4,1.8) to (8,1.2);
  \draw (7,1.6) to (7.4,1.8);
  \draw (7.6,1) to (8,1.2);

  \draw (7.6,1) to (7,1.6);
  \draw (6.4,1.6) to (7,1.6);
  \draw (7,1) to (7.6,1);

  \draw (7,1) to (6.4,1.6);
  \draw (7,1) to (6.5,0.35);
  \draw (6.4,1.6) to (6.5,0.35);

  \draw[dashed,->] (7.5,0.6) to (8.5,0.6);
  \draw (9,0.1) to[bend left] node[below]{$\s_0^+$}(11,0.1);
  \draw[fill=black] (9.2,0.2) node[above]{$\frac12$} circle (0.3ex);
  \draw[fill=black] (10,0.4) node[above]{$\frac12$} circle (0.3ex);
  \draw[fill=black] (10.8,0.2) node[above]{$\frac14$} circle (0.3ex);

   \end{tikzpicture}
   \caption{Schematic picture of the flip $X(2;1)$ to $X^+(2;1)$}\label{fig: schematic}
\end{center}
\end{figure}
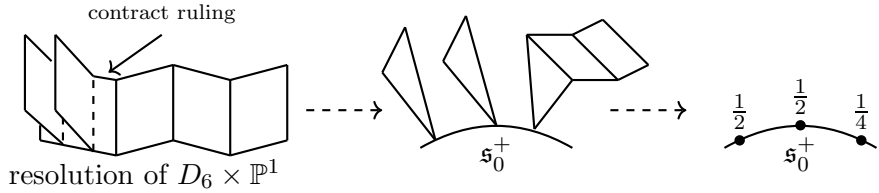
Unlike in Proposition \ref{prop!canonicalmodels}, the canonical linear system of $X(2;1)$ has a fixed part $D_{x_1}$ and $|K_X-D_{x_1}|$ is a basepoint free pencil of $(1,2)$-surfaces. Every fibre has a $D_6$-singularity along the section $\s_0$. On $X^+$, $|K_{X^+}|$ is a pencil with base curve $\s_0^+$. Each element of $|K_X|$ is a $(1,2)$-surface with a $D_6$-singularity where it meets $\s_0$. The flip extracts the central curve $\s_0^+$ giving a partial resolution of the $D_6$-singularity, so each element of $|K_{X^+}|$ has two $A_1$-singularities and one $A_3$-singularity, lying on the base curve $\s_0^+$.

The other two cases have a similar description:
\begin{itemize}
\item After a crepant blowup, $X(3;1)$ has a curve of $E_8$-singularities along $\s_0$ and the flip extracts the curve marked with a square in Figure \ref{fig: E8} below, so the pencil $|K_{X^+}|$ consists of $(1,2)$-surfaces with one $A_1$-singularity and one $A_6$-singularity on the base curve $\s_0^+$.
\item There is also a curve of $E_8$-singularities along $\s_0$ in $X(4;1)$. This time the partial resolution extracts the curve marked with a triangle in Figure \ref{fig: E8} and the elements of $|K_{X^+}|$ have $A_1$, $A_2$ and $A_4$-singularities.
\end{itemize}

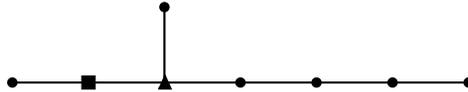
\begin{figure}[ht]
\begin{center} \begin{tikzpicture}[thick, every node/.style = {font = \small}]

  \draw (0,0.2) to (6,0.2);
    \draw[fill=black] (0,0.2)  circle (0.3ex);
    \draw[fill=black] (0.92,0.12)  rectangle (1.08,0.28);
    \draw[fill=black] (1.92,0.12) -- (2.08,0.12) -- (2,0.3);
    \draw[fill=black] (3,0.2)  circle (0.3ex);
   \draw[fill=black] (4,0.2)  circle (0.3ex);
    \draw[fill=black] (5,0.2) circle (0.3ex);
    \draw[fill=black] (6,0.2)  circle (0.3ex);
   \draw (2,0.2) to (2,1.2);
    \draw[fill=black] (2,1.2)  circle (0.3ex);
 \end{tikzpicture}\end{center}
   \caption{Partial resolutions of $X(3;1)$ and $X(4;1)$}\label{fig: E8}
\end{figure}

\newcommand{\etalchar}[1]{$^{#1}$}

\end{document}